\documentclass[12pt]{article}
\usepackage{amsmath,amssymb,amsthm}
\usepackage{hyperref, enumerate}
\usepackage{tikz}
\usepackage{tikz-cd}
\usepackage{graphics}
\usepackage{pinlabel}
\usepackage{yfonts}
\usepackage{float}

\newtheorem{theorem}{Theorem}[section]
\newtheorem{proposition}[theorem]{Proposition}
\newtheorem{corollary}[theorem]{Corollary}
\newtheorem{maintheorem}{Theorem}

\newtheorem{lemma}[theorem]{Lemma}
\newtheorem*{claim}{Claim}
\theoremstyle{definition}
\newtheorem*{definition}{Definition}
\theoremstyle{remark}
\newtheorem*{remark}{Remark}

\DeclareMathOperator{\Mod}{Mod}
\DeclareMathOperator{\Diff}{Diff^{+}}

\DeclareMathOperator{\Fr}{Fr}
\DeclareMathOperator{\SO}{SO}
\DeclareMathOperator{\Tor}{\mathcal{I}}
\DeclareMathOperator{\Out}{Out}
\DeclareMathOperator{\Aut}{Aut}

\DeclareMathOperator{\FS}{FS(\mathbb{Z}^n)}
\DeclareMathOperator{\id}{id}
\DeclareMathOperator{\HD}{HD}
\DeclareMathOperator{\CD}{CD}
\DeclareMathOperator{\BCD}{BCD}
\DeclareMathOperator{\PD}{PD}
\DeclareMathOperator{\interior}{int}
\DeclareMathOperator{\Push}{Push}
\DeclareMathOperator{\PushFrame}{\widetilde{\Push}}
\DeclareMathOperator{\GL}{GL}
\DeclareMathOperator{\Hom}{Hom}

\newcommand{\IO}{IO}
\newcommand{\IA}{IA}
\newcommand{\XX}[2]{M_{#1,#2}}
\newcommand{\XXnobound}[1]{M_{#1}}
\newcommand{\Outb}[2]{\mathrm{Out}(F_{#1,#2})}
\newcommand{\IOpar}[3]{IO_{#1,#2}^{#3}}
\newcommand{\sphtwist}[2]{\text{STwist}(\XX{#1}{#2})}
\newcommand{\sphtwistnobound}[1]{\text{STwist}(\XXnobound{#1})}

\DeclareMathOperator{\Int}{int}

\DeclareMathOperator{\nonsepcpx}{\mathbb{S}_n^{\text{ns}}}

\allowdisplaybreaks

\begin{document}

\title{Cutting and Pasting in the Torelli subgroup of $\Out(F_n)$}
\author{Jacob Landgraf\footnote{This material is based upon work supported by the National Science Foundation under Grant Number DMS-1547292}}
\date{}
\maketitle

\begin{abstract}
  \noindent Using ideas from 3-manifolds, Hatcher--Wahl defined a notion of automorphism groups of free groups with boundary. We study their Torelli subgroups, adapting ideas introduced by Putman for surface mapping class groups. Our main results show that these groups are finitely generated, and also that they satisfy an appropriate version of the Birman exact sequence.
\end{abstract}

\section{Introduction}
Let $F_n = \langle x_1, \ldots, x_n \rangle$ be the free group on $n$ letters, and let $\Out(F_n)$ be the group of outer automorphisms of $F_n$. In many ways, $\Out(F_n)$ behaves very similarly to $\Mod(\Sigma_{g,b})$, the mapping class group of the surface $\Sigma_{g,b}$ of genus $g$ with $b$ boundary components. For an overview of some of these similarities, see \cite{BridsonVogtmann}. 

One such connection is that they both contain a Torelli subgroup. In the mapping class group, the Torelli subgroup $\Tor(\Sigma_{g,b}) \subset \Mod(\Sigma_{g,b})$ is defined to be the kernel of the action on $H_1(\Sigma_{n,b}; \mathbb{Z})$ for $b=0,1$. In $\Out(F_n)$, we define a similar subgroup\footnote{It is also common to see this group denoted by $\IA_n$, but we wish to reserve this notation for the analogous subgroup of $\Aut(F_n)$}, denoted $\IO_n$, as the kernel of the action of $\Out(F_n)$ on $H_1(F_n; \mathbb{Z}) = \mathbb{Z}^n$. 

On surfaces with multiple boundary components, there are many possible definitions one might use to define a Torelli subgroup of $\Mod(\Sigma_{g,b})$. In \cite{CutPaste}, Putman defines a Torelli subgroup $\Tor(\Sigma_{g,b}, P)$ for $b>1$ requiring the additional data of a partition $P$ of the boundary components. The goal of the current paper is to mirror Putman's procedure to define an ``$\IO_n$ with boundary.''
 


Let $\XX{n}{b} = \#_n (S^1 \times S^2) \setminus (b \text{ open 3-disks})$. For simplicity, we will write $\XXnobound{n}$ if $b=0$. A key property of $\XX{n}{b}$ is that it has fundamental group $F_n$. Fix such an identification. The \emph{mapping class group} $\Mod(\XX{n}{b})$ is the group of orientation-preserving diffeomorphisms of $\XX{n}{b}$ fixing the boundary pointwise modulo isotopies fixing the boundary pointwise. Letting $\Diff(\XX{n}{b}, \partial\XX{n}{b})$ be the topological group of diffeomorphisms fixing the boundary pointwise, we can also write $\Mod(\XX{n}{b}) = \pi_0(\Diff(\XX{n}{b}, \partial\XX{n}{b}))$. By a theorem of Laudenbach \cite{Laudenbach}, there is an exact sequence
\begin{equation}\label{Laudseq}
1 \to (\mathbb{Z}/2)^{n} \to \Mod(\XXnobound{n}) \to \Out(F_n) \to 1,
\end{equation}
where the map $\Mod(\XXnobound{n}) \to \Out(F_n)$ is given by the action (up to conjugation) on $\pi_1(\XXnobound{n})$, and the $(\mathbb{Z}/2)^{n}$ is generated by sphere twists about $n$ disjointly embedded $2$-spheres (see Section \ref{prelimsection} for the definition and relevant properties of sphere twists). Recent work of Brendle, Broaddus, and Putman \cite{LaudenbachSplits} shows that this sequence actually splits as a semidirect product. This exact sequence implies that, modulo a finite group, $\Out(F_n)$ acts on $\XXnobound{n}$ up to isotopy. Therefore, $\XXnobound{n}$ plays almost the same role for $\Out(F_n)$ that $\Sigma_{g,b}$ plays for $\Mod(\Sigma_{g,b})$.

\paragraph{Adding boundary components.} From Laudenbach's sequence (\ref{Laudseq}), we see that $\Out(F_n) \cong \Mod(\XXnobound{n})/\sphtwistnobound{n}$, where $\sphtwistnobound{n} \cong (\mathbb{Z}/2)^n$ is the subgroup of $\Mod(\XXnobound{n})$ generated by sphere twists. Now that we have related $\Out(F_n)$ to a geometrically defined group, we can start introducing boundary components. Extending the relationship given by Laudenbach's sequence, we define ``$\Out(F_n)$ with boundary" as
\begin{equation*}
\Outb{n}{b} := \Mod(\XX{n}{b})/\sphtwist{n}{b}.
\end{equation*}
When $b=1$, Laudenbach \cite{Laudenbach} also shows that $\Outb{n}{1} \cong \Aut(F_n)$. Hatcher-Wahl \cite{HatcherWahl} introduced a more general version of $\Outb{n}{b}$, which they denoted by $A_{n,k}^s$. The original definition of $A_{n,k}^s$ has to do with classes of self-homotopy equivalences of a certain graph. However, in \cite{HatcherWahl} the authors give an equivalent definition, which says that $A_{n,k}^s$ is the mapping class group of $\XXnobound{n}$ with $s$ spherical and $k$ toroidal boundary components, modulo sphere twists. With this definition, we see that $\Outb{n}{b} = A_{n,0}^b$. Similar groups have been examined in the work of Jensen-Wahl \cite{JensenWahl} and Wahl \cite{Wahl}. Their versions, however, involve only toroidal boundary components, and thus are distinct from $\Outb{n}{b}$. 

\paragraph{Torelli subgroups.} An important feature of sphere twists (discussed in Section \ref{prelimsection}) is that they act trivially on homotopy classes of embedded loops, and thus act trivially on $H_1(\XXnobound{n})$. Therefore, the action of $\Mod(\XX{n}{b})$ on $H_1(\XX{n}{b})$ induces an action of $\Outb{n}{b}$ on $H_1(\XX{n}{b})$. We can then define the Torelli subgroup $ \IOpar{n}{b}{} \subset \Outb{n}{b}$ to be the kernel of this action. However, this definition does not capture all homological information when $b>1$, especially when $\XX{n}{b}$ is being embedded in $\XX{m}{c}$. To see why, consider the scenario depicted in Figure \ref{torelliproblem}, in which $\XX{2}{2}$ has been embedded into $\XXnobound{4}$. This embedding induces a homomorphism $\iota_M: \Mod(\XX{2}{2}) \to \Mod(\XXnobound{4})$ obtained by extending by the identity. This map sends sphere twists to sphere twists, and so we get an induced map $\iota_*: \Outb{2}{2} \to \Out(F_4)$. However, this does \emph{not} restrict to a map $\IOpar{2}{2}{} \to \IO_4$ under this definition of $\IOpar{n}{b}{}$ since elements of $\IOpar{2}{2}{}$ are not required to fix the homology class of the subarc of $\alpha$ lying inside $\XX{2}{2}$. To address this issue, we will use a slightly modified homology group. 
\begin{definition}
  Fix a partition $P$ of the boundary components of $\XX{n}{b}$.
\begin{enumerate}[(a)]
  \item Two boundary components $\partial_1, \partial_2$ of $\XX{n}{b}$ are \emph{$P$-adjacent} if there is some $p \in P$ such that $\{\partial_1, \partial_2\} \subset p$. 
  \item Let $H_1^P(\XX{n}{b})$ be the subgroup of $H_1(\XX{n}{b}, \partial\XX{n}{b})$ spanned by
  \begin{align*}
  \{  [h] \in H_1(\XX{n}{b}, \partial\XX{n}{b}) \mid &\text{ either $h$ is a simple closed curve or} \\
  &\text{ $h$ is a properly embedded arc with endpoints } \\
  &\text{ in distinct $P$-adjacent boundary components}  \}.
  \end{align*}

  \item There is a natural action of $\Outb{n}{b}$ on $H_1^P(\XX{n}{b})$, and we define the Torelli subgroup $\IOpar{n}{b}{P} \subset \Outb{n}{b}$ to be the kernel of this action.
\end{enumerate}
\end{definition}

Returning to Figure \ref{torelliproblem}, let $P$ be the trivial partition of the boundary components of $\XX{2}{2}$ with a single $P$-adjacency class. With this choice of partition, we see that $[\alpha \cap \XX{2}{2}] \in H_1^P(\XX{2}{2})$. If $f \in \IOpar{2}{2}{P}$, then it follows that $\iota_*(f) \in \Out(F_4)$ preserves the homology class of $\alpha$. Therefore, $\iota_*(f) \in \IO_4$, and so $\iota_*$ restricts to a map $\IOpar{2}{2}{P} \to \IO_4$.

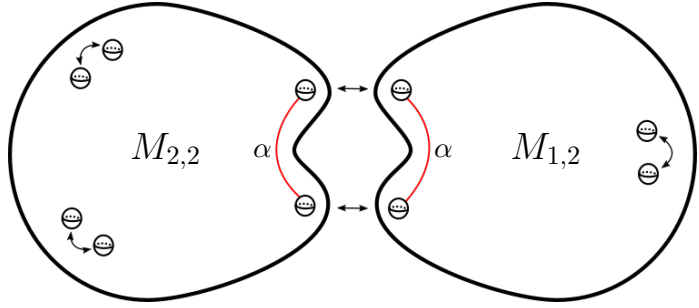
\begin{figure}
  \centering
  \labellist
  \small\hair 2pt
  \pinlabel $\alpha$ at 210 125
  \pinlabel $\alpha$ at 360 125
  \large
  \pinlabel $\XX{2}{2}$ at 130 125
  \pinlabel $\XX{1}{2}$ at 445 125
  \endlabellist
  \includegraphics[width=0.6\textwidth]{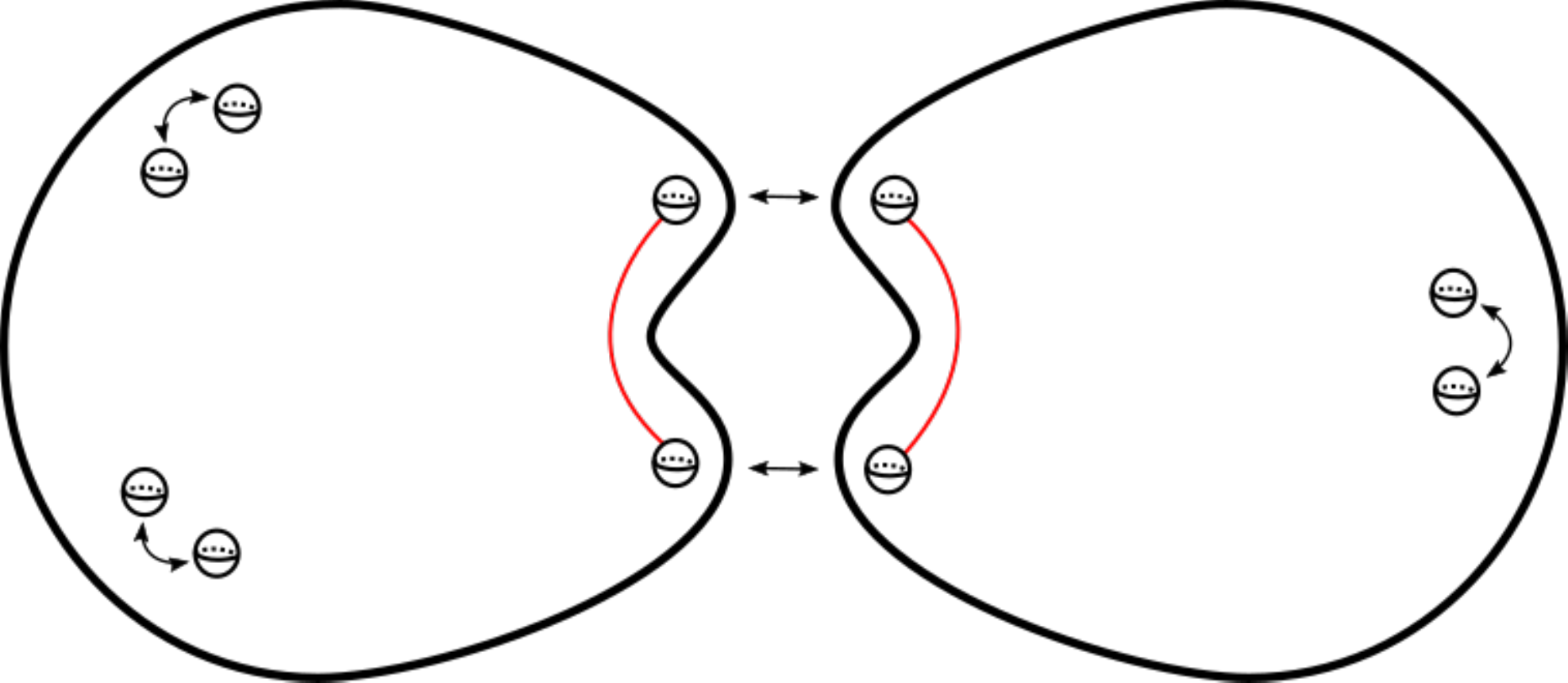}
  \caption{A copy of $\XX{2}{2}$ and $\XX{1}{2}$ glued together to obtain $\XXnobound{4}$. We realize $\XX{2}{2}$ as a 3-sphere with the six indicated open balls removed, then the boundaries of these removed balls are identified according to the arrows (and similarly for $\XX{1}{2}$). The class $[\alpha]$ need not be fixed by elements of $\IOpar{2}{2}{}$ with the na\"{i}ve definition.}\label{torelliproblem}
\end{figure} 
  
\paragraph{Restriction.} As we discussed in the last paragraph, given an embedding $\iota: \XX{n}{b} \hookrightarrow \XXnobound{m}$, we can extend by the identity to get a map $\iota_*: \Outb{n}{b} \to \Out(F_m)$. In general, $\iota_*$ may not be injective. However, it is injective if no connected component of $\XXnobound{m} \setminus \XX{n}{b}$ is diffeomorphic to $D^3$ (see Appendix \ref{injectivity}). Moreover, such an embedding induces a natural partition of the boundary components of $\XX{n}{b}$ as follows. 
\begin{definition}
  Fix an embedding $\iota: \XX{n}{b} \hookrightarrow \XXnobound{m}$. Let $N$ be a connected component of $\XXnobound{m} \setminus \interior(\XX{n}{b})$, and let $p_N$ be the set of boundary components of $\XX{n}{b}$ shared with $N$. Then the partition $P$ of the boundary components of $\XX{n}{b}$ induced by $\iota$ is defined to be 
  \begin{equation*}
  P = \{ p_N \mid N \text{ a connected component of } \XXnobound{m} \setminus \XX{n}{b}\}.
  \end{equation*}
\end{definition}
\noindent With this definition, one might guess that $\iota_*^{-1}(\IO_n) = \IOpar{n}{b}{P}$. This turns out to be the case, and this is our first main theorem, which we prove in Section \ref{restrictionsection}.

\begin{maintheorem}[Restriction Theorem]\label{restriction}
Let $\iota: \XX{n}{b} \hookrightarrow \XXnobound{m}$ be an embedding, $\iota_*: \Outb{n}{b} \to \Out(F_m)$ the induced map, and $P$ the induced partition of the boundary components of $\XX{n}{b}$. Then $\IOpar{n}{b}{P} = \iota_*^{-1}(\IO_m)$.
\end{maintheorem}

\paragraph{Birman exact sequence.} From here, we move on to exploring the parallels between these Torelli subgroups and those of mapping class groups. There is a well-known relationship between the mapping class groups of surfaces with a different number of boundary components called the Birman exact sequence (see \cite{primer}):
\begin{equation*}
1 \to \pi_1(UT(\Sigma_{n,b-1})) \to \Mod(\Sigma_{g,b}) \to \Mod(\Sigma_{g,b-1}) \to 1.
\end{equation*}
Here, $UT(\Sigma_{n,b-1})$ is the unit tangent bundle of $\Sigma_{n,b-1}$, the map $\pi_1(UT(\Sigma_{n,b-1})) \to \Mod(\Sigma_{g,b})$ is given by pushing a boundary component around a loop, and the map $\Mod(\Sigma_{g,b}) \to \Mod(\Sigma_{g,b-1})$ is given by attaching a disk onto this boundary component. In Section \ref{birmansection}, we will prove versions of the Birman exact sequence for $\Mod(\XX{n}{b})$ and $\Outb{n}{b}$, all culminating in the following sequence for $\IOpar{n}{b}{P}$. 

\begin{maintheorem}[Birman exact sequence]\label{birman}
  Fix $n, b > 0$ such that $(n,b) \neq (1,1)$, and let $\XX{n}{b} \hookrightarrow \XX{n}{b-1}$ be an embedding obtained by gluing a ball to the boundary component $\partial$. Fix $x \in \XX{n}{b-1} \setminus \XX{n}{b}$. Let $P$ be a partition of the boundary components of $\XX{n}{b}$, let $P'$ be the induced partition of the boundary components of $\XX{n}{b-1}$, and let $p \in P$ be the set containing $\partial$. We then have an exact sequence
  \begin{equation*}
  1 \to L \to \IOpar{n}{b}{P} \overset{\iota_*}{\to} \IOpar{n}{b-1}{P'} \to 1,
  \end{equation*}
  where $L$ is equal to:
  \begin{enumerate}[(a)]
    \item $\pi_1(\XX{n}{b-1}, x) \cong F_n$ if $p = \{\partial\}$.
  
    \item $[\pi_1(\XX{n}{b-1}, x), \pi_1(\XX{n}{b-1}, x)] \cong [F_n, F_n]$ if $p \neq \{ \partial \}$.
  \end{enumerate}
  Moreover, this sequence splits if $b \geq 2$. 
\end{maintheorem} 

\begin{remark}
  This theorem may seem superficially similar to results proven by Day-Putman in \cite{DayPutmanAut} and \cite{DayPutmanTor}. However, we consider a very different notion of ``automorphisms with boundary," and so these results are unrelated.
\end{remark}
  
\paragraph{Finite generation.} Once we have established this version of the Birman exact sequence, in Section \ref{generatorssection}, we will define a generating set for $\IOpar{n}{b}{P}$. This generating set will be inspired by the generating set for $\IO_n$ found by Magnus \cite{magnus} in 1935.

\begin{theorem}[Magnus]\label{magnus}
  Let $F_n = \langle x_1, \ldots, x_n \rangle$. The group $IO_n$ is generated by the $\Out(F_n)$-classes of the automorphisms
  \begin{equation*}
  M_{ij}: x_i \mapsto x_jx_ix_j^{-1}, \qquad M_{ijk}: x_i \mapsto x_i[x_j,x_k],
  \end{equation*}
  for all distinct $i,j,k \in \{1, \ldots, n\}$ with $j < k$. Here, the automorphisms are understood to fix $x_\ell$ for $\ell \neq i$.
\end{theorem}

Throughout this paper, we will use the convention $[a,b] = aba^{-1}b^{-1}$. Since we defined $\IOpar{n}{b}{P}$ to be a subgroup of $\Mod(\XX{n}{b})/\sphtwist{n}{b}$, our generators will be defined geometrically rather than algebraically. However, in the case of $b=0$, they will reduce directly to Magnus's generators. In Section \ref{finitegenerationsection}, we will show that these elements do indeed generate $\IOpar{n}{b}{P}$.

\begin{maintheorem}\label{finitegeneration}
The group $\IOpar{n}{b}{P}$ is finitely generated for $n \geq 1$, $b \geq 0$.
\end{maintheorem}

This is rather striking because the analogous result for Torelli subgroups of mapping class groups with multiple boundary components is still open. We will prove this theorem by using the Birman exact sequence to reduce to $b=0$ and applying Magnus's theorem. Unfortunately, the tools we have constructed do not seem strong enough to give a novel proof of Magnus's theorem. We will, however, prove a weaker version in Section \ref{magnussection}. The original proof of Magnus's Theorem \ref{magnus} comes in two steps: showing that the given automorphisms $\Out(F_n)$-normally generate $\IO_n$, and then showing that the subgroup they generate is normal in $\Out(F_n)$. We will give a proof of the first step in our setting. For alternative proofs of the first step, as well as more information on the second step in this context, see \cite{BBM} and \cite{PartialBases}.

\begin{maintheorem}\label{magnusweak}
The group $\IO_n$ is $\Out(F_n)$-normally generated by the automorphisms $M_{ij}$ and $M_{ijk}$, where $i,j,k \in \{1, \ldots, n\}$ and $j<k$.  
\end{maintheorem}
 
\paragraph{Abelianization.} Once we have a finite generating set for $\IOpar{n}{b}{P}$, a natural question arises: how does the cardinality of this generating set compare to the rank of $H_1(\IOpar{n}{b}{P})$? For $b \leq 1$, this question is answered by a result of Andreadakis \cite{Andreadakis} and Bachmuth \cite{Bachmuth}.

\begin{theorem}[Andreadakis, Bachmuth]\label{rankabel}
The abelianization of $\IOpar{n}{b}{}$ is torsion-free of rank $n \cdot \binom{n}{2} - n$ if $n=0$, and rank $n \cdot \binom{n}{2}$ if $n=1$. 
\end{theorem}
 
\noindent This theorem was proved using a version of the Johnson homomorphism $$\tau: \IA_n \to \Hom(H, \wedge^2 H),$$ where $H = H_1(F_n) = \mathbb{Z}^n$. We will recall the definition of this homomorphism in Section \ref{abelianizationsection}, along with the proof of Theorem \ref{rankabel}. We then move on to computing the rank of $H_1(\IOpar{n}{b}{P})$ for $b > 1$. To do this, we choose an embedding $\XX{n}{b} \hookrightarrow \XX{m}{1}$, which induces an injection $\IOpar{n}{b}{P} \to \IOpar{m}{1}{} = \IA_m$. Composing this map with $\tau$ gives a map $\tau_*: \IOpar{n}{b}{P} \to \Hom(H, \wedge^2 H)$. We then compute the image of our generators under $\tau_*$, and use this to count the rank of $\tau_*(\IOpar{n}{b}{P})$. 

\begin{maintheorem}\label{abelianizationtheorem}
  The abelianization of $\IOpar{n}{b}{P}$ is torsion-free of rank
  \begin{equation*}
    n \cdot \binom{n}{2} + \left(b \cdot \binom{n}{2} - \vert P \vert \cdot \binom{n}{2}\right) + (\vert P \vert \cdot n - n).
  \end{equation*}
\end{maintheorem}

\paragraph{Acknowledgements.} I would like to thank my advisor Andy Putman for directing me to $\Out(F_n)$ and its Torelli subgroup, and for his input during the revision process. I would also like to thank Patrick Heslin and Aaron Tyrrell for helpful conversations regarding diffeomorphism groups, as well as Dan Margalit for an enlightening question which resulted in the addition of Section \ref{abelianizationsection}.

\paragraph{Outline.} In section \ref{prelimsection}, we will give a short overview of sphere twists. We then move on to proving Theorem \ref{restriction} in Section \ref{restrictionsection}. We will establish all of our versions of the Birman exact sequence (including Theorem \ref{birman}) in Section \ref{birmansection}. In Section \ref{generatorssection}, we will define our candidate generators for $\IOpar{n}{b}{P}$, and we will prove that they generate (Theorem \ref{finitegeneration}) in Section \ref{finitegenerationsection} using the Birman exact sequence and Magnus's Theorem \ref{magnus}. In Section \ref{magnussection}, we will prove Theorem \ref{magnusweak}. We then move on to Section \ref{abelianizationsection}, in which we recall the definition of the Johnson homomorphism for $\IA_n$, and use it to compute the rank of the abelianization of $\IOpar{n}{b}{P}$, proving Theorem \ref{abelianizationtheorem}. Finally, we conclude with two appendices. In Appendix \ref{injectivity}, we provide conditions for a map $\Outb{n}{b} \to \Out(F_m)$ induced by an inclusion to be injective, and in Appendix \ref{homologybasessection} we prove a lemma which allows us to realize bases of $H_2(\XXnobound{m})$ as collections of disjoint oriented spheres.

\paragraph{Figure conventions.} We will frequently direct the reader to figures which are intended to give some geometric intuition for the manifold $\XX{n}{b}$. In order to assemble $\XX{n}{b}$, we begin with one or more copies of $S^3$, remove a collection of open balls, and then glue the resulting boundary components together in pairs. These gluings will be indicated by double-sided arrows connecting the boundary spheres being glued. As an example, see Figure \ref{figconvention}.

\begin{figure}
  \centering
  \labellist
  \small\hair 2pt
  \pinlabel $\XX{1}{2}$ at 88 88 
  \pinlabel $\XX{1}{3}$ at 290 88
  \pinlabel $\XX{2}{1}$ at 500 88
  \endlabellist
  \includegraphics[width=\textwidth]{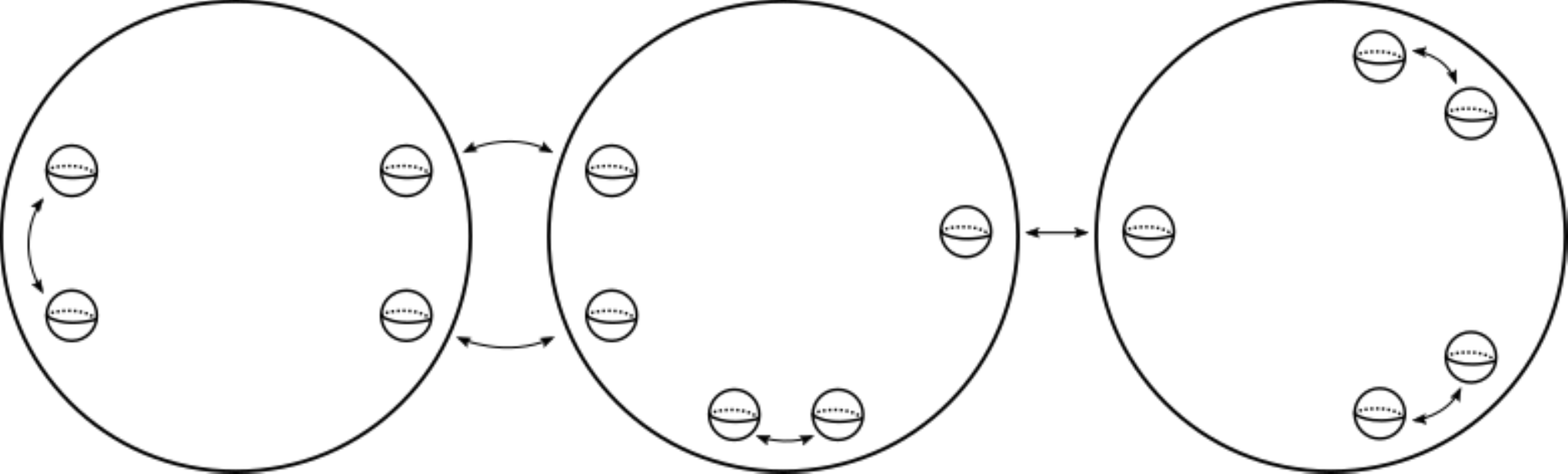}
  \caption{$\XXnobound{5}$ realized by gluing $\XX{1}{2}$, $\XX{1}{3}$, and $\XX{2}{1}$ together along their boundaries as indicated by the arrows.}
  \label{figconvention}
\end{figure}

\section{Preliminaries}\label{prelimsection}

Since sphere twists play a fundamental role throughout the remainder of the paper, we will give a brief overview of them here.

\paragraph{Sphere twists.} Fix a smoothly embedded 2-sphere $S \subset \XX{n}{b}$, and let $U \cong S \times [0,1]$ be a tubular neighborhood of $S$. Recall that $\pi_1(\SO(3), \id) \cong \mathbb{Z}/2\mathbb{Z}$, and the nontrivial element $\gamma: [0,1] \to \SO(3)$ is given by rotating $\mathbb{R}^3$ one full revolution about any fixed axis through the origin. Fix an identification $S = S^2 \subset \mathbb{R}^3$. Then, we define the sphere twist about $S$, denoted $T_S \in \Mod(\XX{n}{b})$, to be the class of the diffeomorphism which is the identity on $\XX{n}{b} \setminus U$ and is given by $(x, t) \mapsto (\gamma(t) \cdot x, t)$ on $U \cong S \times [0,1]$. The isotopy class of $T_S$ depends only on the isotopy class of $S$. In fact, more is true: Laudenbach \cite{Laudenbach} showed that the class of $T_S$ depends only on the homotopy class of $S$.

\paragraph{Action on curves and surfaces.} Since $\pi_1(\SO(3), \id) \cong \mathbb{Z}/2\mathbb{Z}$, we see that sphere twists have order at most two. However, it is tricky to show that sphere twists are actually nontrivial because they act trivially on homotopy classes of embedded arcs and surfaces. To see why this is true, let $S \subset \XX{n}{b}$ be an embedded 2-sphere, and let $U = S \times [0,1]$ be a tubular neighborhood of $S$. Suppose that $\alpha$ is an arc or surface embedded in $\XX{n}{b}$. We can homotope $\alpha$ such that it is either disjoint from $U$ or intersects $U$ transversely. Let $p \in S$ be one of points in $S$ which lies on the axis of rotation used to construct $T_S$. We can homotope $\alpha$ such that $\alpha \cap U$ collapses into $p \in [0,1]$. Note that this process is not an isotopy, and $\alpha$ is no longer embedded in $\XX{n}{b}$. This is not an issue because a result of Laudenbach \cite{Laudenbach} shows that if $\alpha$ is fixed up to homotopy, then it is fixed up to isotopy. Since $T_S$ fixes $p \times [0,1]$ pointwise, it follows that $T_S$ fixes $\alpha$ up to homotopy. The upshot of this is that a more sophisticated invariant must be constructed to detect the nontriviality of $T_S$. In \cite{Laudenbach2, Laudenbach}, Laudenbach uses framed cobordisms to show that for $b = 0,1$, the sphere twist $T_S$ is trivial if and only if $S$ is separating. In the case of no boundary components, Brendle, Broaddus, and Putman \cite{LaudenbachSplits} give another proof of this fact by showing that sphere twists act nontrivially on a trivialization of the tangent bundle of $\XXnobound{n}$ up to isotopy. 

\paragraph{Sphere twist subgroup.} Let $\sphtwist{n}{b} \subset \Mod(\XX{n}{b})$ be the subgroup generated by sphere twists. Given $\mathfrak{f} \in \Mod(\XX{n}{b})$ and a sphere twists $T_S$, we have the ``change of coordinates" formula
\begin{equation*}
  \mathfrak{f} T_S \mathfrak{f}^{-1} = T_{\mathfrak{f}(S)}.
\end{equation*}
This shows that $\sphtwist{n}{b}$ is a normal subgroup of $\Mod(\XX{n}{b})$. In fact, even more is true. Letting $\mathfrak{f} = T_{S'}$ in the above formula and using the fact that sphere twists act trivially on embedded surfaces up to isotopy, we find that
\begin{equation*}
T_{S'} T_S T_{S'}^{-1} = T_{T_{S'}(S)} = T_S,
\end{equation*}
which implies $\sphtwist{n}{b}$ is actually abelian. Since nontrivial sphere twists have order two, it follows that $\sphtwist{n}{b}$ is isomorphic to a product of copies of $\mathbb{Z}/2\mathbb{Z}$. For $b = 0,1$, another result of Laudenbach shows that $\sphtwist{n}{b} \cong (\mathbb{Z}/2\mathbb{Z})^n$ and is generated by the sphere twists about the $n$ core spheres $* \times S^2$ in each $S^1 \times S^2$ summand. For $b > 1$, one can show that $\sphtwist{n}{b} \cong (\mathbb{Z}/2\mathbb{Z})^{n + b - 1}$. The $-1$ in the exponent reflects the fact that the product of all the sphere twists about boundary components is trivial. Since we will need this fact later, we include a proof here.

\begin{figure}
  \centering
  \labellist
  \small\hair 2pt
  \pinlabel $z$ at 310 680
  \pinlabel $S_1$ at 60 395
  \pinlabel $S_2$ at 350 490
  \pinlabel $S_3$ at 350 300
  \pinlabel $S_4$ at 350 110
  \endlabellist
  \includegraphics[width=0.3\textwidth]{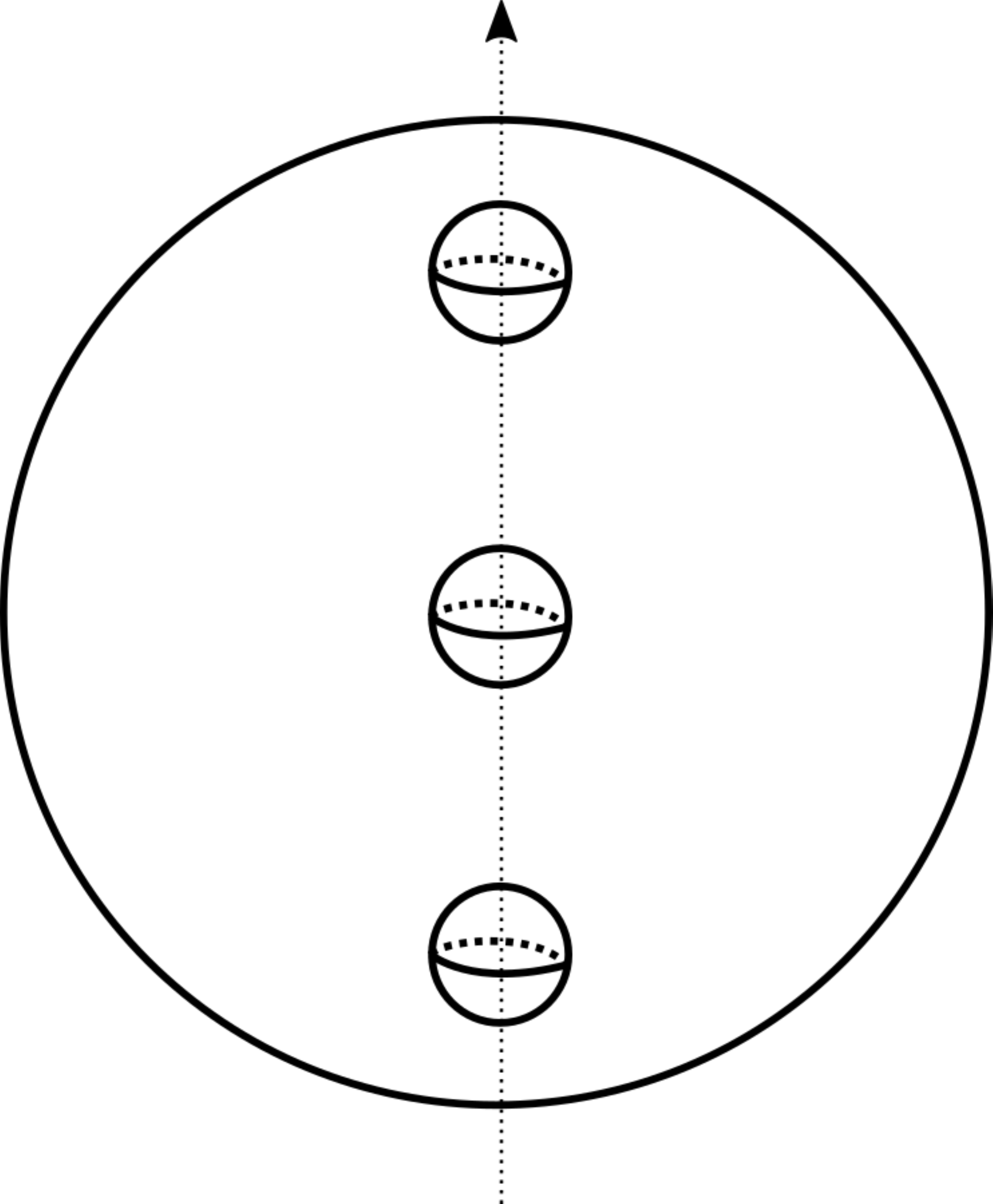}
  \caption{$\XX{0}{4}$ embedded in $\mathbb{R}^3$.}\label{embedr3}
\end{figure}
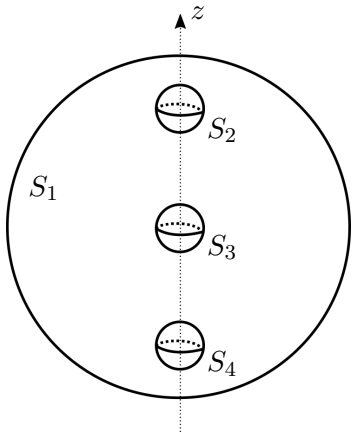 

\begin{lemma}\label{boundaryprodtrivial}
If $S_1, \ldots, S_b \subset \XX{n}{b}$ be spheres parallel to the $b$ boundary components of $\XX{n}{b}$, then the element $T_{S_1} \cdots T_{S_b}$ is trivial in $\Mod(\XX{n}{b})$. 
\end{lemma}

\begin{proof}
  We will prove this by induction on $n$. As the base case, consider $\XX{0}{b}$. The argument in this case follows a proof of Hatcher and Wahl \cite[Pg. 214-215]{HatcherWahl3-Man}, but we include the proof here as well for completeness. If $b=0$, then the statement is trivial. If $b > 0$, then we can embed $\XX{0}{b}$ in $\mathbb{R}^3$ as the unit ball with $b-1$ smaller balls removed along the $z$-axis (see Figure \ref{embedr3}). We may then use the $z$-axis as the axis of rotation for the sphere twists about all the boundary components. Taking $S_1$ to be the unit sphere, we then see that the product  $T_2 \cdots T_b$ is isotopic to $T_1$. Since sphere twists have order two, this gives the desired relation, and so we have completed the base case. 

  Next, consider $\XX{n}{b}$ for $n > 0$. Since $n > 0$, there exists a nonseparating sphere $S \subset \XX{n}{b}$ which is disjoint from $S_1, \ldots, S_b$. Splitting $\XX{n}{b}$ along $S$ yields a submanifold diffeomorphic to $\XX{n-1}{b+2}$. Let $\iota_M: \Mod(\XX{n-1}{b+2}) \to \Mod(\XX{n}{b})$ be the map induced by inclusion. Let $T_1, \ldots, T_{b+2}$ be the sphere twists about the boundary components of $\XX{n-1}{b+2}$, and order them such that $\iota_M(T_j) = T_{S_j}$ for $0 \leq j \leq b$. With this ordering, notice that $\iota_M(T_{b+1}) = \iota_M(T_{b+2}) = T_S$. Since sphere twists have order two,
  \begin{equation*}
    \iota_M(T_1 \cdots T_{b+2}) = T_{S_1} \cdots T_{S_b} \cdot T_S^2 = T_{S_1} \cdots T_{S_b}.
  \end{equation*}
  By our induction hypothesis, $T_1 \cdots T_{b+2}$ is trivial in $\Mod(\XX{n-1}{b+2})$, and so we are done. 
\end{proof}

If $b=1$, this shows that the sphere twist about the boundary component is trivial. However, if $b > 1$, then the sphere twists about boundary components are nontrivial. We will also need this fact, so we prove it here.

\begin{lemma}\label{boundarytwistnontriv}
Let $b > 1$, and let $\partial$ be a boundary component of $\XX{n}{b}$. Then $T_\partial \in \Mod(\XX{n}{b})$ is nontrivial.
\end{lemma}

\begin{proof}
  Let $\partial'$ be a boundary component of $\XX{n}{b}$ different from $\partial$. Then we get an embedding $\iota : \XX{n}{b} \hookrightarrow \XXnobound{n+1}$ by attaching $\partial$ and $\partial'$ with a copy of $S^2 \times I$, and capping off all the remainding boundary components. Let $\iota_M: \Mod(\XX{n}{b}) \to \Mod(\XXnobound{n+1})$ be the map induced by $\iota$. Then $\iota_M(T_\partial)$ is a sphere twist about about a nonseparating sphere. Earlier in this section, we saw that such sphere twists are nontrivial, and so we conclude that $T_\partial$ is nontrivial as well. 
\end{proof}




\section{Restriction Theorem}\label{restrictionsection}

Fix $n, b \geq 0$, and let $P$ be a partition of the boundary components of $\XX{n}{b}$. Recall that we have defined $H_1^P(\XX{n}{b})$ to be the submodule of $H_1(\XX{n}{b}, \partial\XX{n}{b})$ generated by 
\begin{align*}
  \{  [h] \in H_1(\XX{n}{b}, \partial\XX{n}{b}) \mid &\text{ either $h$ is a simple closed curve or} \\
  &\text{ $h$ is a properly embedded arc with endpoints} \\
  &\text{ in distinct $P$-adjacent boundary components}  \},
\end{align*}
and $\IOpar{n}{b}{P}$ is the kernel of the action of $\Outb{n}{b}$ on $H_1^P(\XX{n}{b})$ induced by the action of $\Mod(\XX{n}{b})$. 

\begin{remark}
  This version of homology is simpler than the one used in \cite{CutPaste}. There are two reasons for this.
  
  \begin{itemize}
    \item In our case, we can take homology relative to the entire boundary, whereas in \cite{CutPaste}, homology is taken relative to a set consisting of a single point from each boundary component. This is because in surfaces, the boundary components give nontrivial elements of $H_1$, and the arcs considered in $H_1^P(\Sigma_{g,b})$ can get ``wrapped around" those boundary components. This is not a problem in our setting because loops in boundary components of $\XX{n}{b}$ are trivial in $H_1$.
   
    \item Next, suppose we have an embedding $\Sigma_{g,b} \hookrightarrow \Sigma_{g'}$ of surfaces. It is possible for a nontrivial element $a \in H_1(\Sigma_{g,b})$ to become trivial in $H_1(\Sigma_{g'})$ (for instance, if a boundary component is capped off). So, there could be elements of $\Mod(\Sigma_{g,b})$ which act trivially on $H_1(\Sigma_{g'})$, but not fix $a$. In other words, the Torelli group would not be closed under restrictions. To fix this, the author in \cite{CutPaste} must mod out by the submodules of $H_1(\Sigma_{g,b})$ spanned by the $p \in P$ (with proper orientation chosen). This is not a problem in the 3-dimensional case however, since an inclusion $\XX{n}{b} \hookrightarrow \XXnobound{m}$ induces an injection on homology.
  \end{itemize}
  \end{remark}
  
We can now move on to the proof of Theorem \ref{restriction}.

\begin{proof}[Proof of Theorem \ref{restriction}]
Let $\iota: \XX{n}{b} \hookrightarrow \XXnobound{m}$ be an embedding, and let $\iota_*: \Outb{n}{b} \to \Out(F_m)$ be the induced map. Recall that we must show that $\iota_*^{-1}(\IO_m) = \IOpar{n}{b}{P}$, where $P$ is the partition of the boundary components induced by $\iota$ as described in the introduction. 

This proof will follow the proof of \cite[Theorem~3.3]{CutPaste}. Define the following subsets of $H_1(\XXnobound{m})$ (we use $\cdot$ to denote concatenation of arcs):
\begin{align*}
  Q_1 = \{ [h] \in H_1(\XXnobound{m}) \mid &\text{ $h$ is a simple closed curve in $\XXnobound{m} \setminus \XX{n}{b}$} \} \\
  Q_2 = \{ [h] \in H_1(\XXnobound{m}) \mid &\text{ $h$ is a simple closed curve in $\XX{n}{b}$} \} \\
  Q_3 = \{ [h_1 \cdot h_2] \in H_1(\XXnobound{m}) \mid &\text{ $h_1$ is a properly embedded arc in $\XX{n}{b}$ with} \\
  &\text{ endpoints in distinct $P$-adjacent boundary} \\
  &\text{ components and $h_2$ is a properly embedded arc } \\
  &\text{ in $\XXnobound{m} \setminus \XX{n}{b}$ with the same endpoints as $h_1$} \}.
\end{align*}
We claim that the homology group $H_1(\XXnobound{m})$ is spanned by $Q_1 \cup Q_2 \cup Q_3$. To see why, let $[\alpha] \in H_1(\XXnobound{m})$ be the class of a loop $\alpha$. If $\alpha$ can be homotoped to lie entirely inside $\XX{n}{b}$ or $\XXnobound{m} \setminus \XX{n}{b}$, then we are done. On the other hand, suppose that crosses the boundary of $\XX{n}{b}$. Without loss of generality, we may assume that $\alpha$ crosses the boundary of $\XX{n}{b}$ exactly twice since any loop can be surgered into a collection of such loops (see Figure \ref{surger}). It follows that $\alpha$ has the form $\alpha = \gamma \cdot \delta$, where $\gamma \subset \XX{n}{b}$ is an arc connecting boundary components of $\XX{n}{b}$, and $\delta \subset \XXnobound{m} \setminus \XX{n}{b}$ is a arc with the same endpoints as $\gamma$. Recall that under the partition $P$ induced by the inclusion $\iota$, two boundary components are $P$-adjacent if they lie on the same component of $\XXnobound{m} \setminus \XX{n}{b}$. Therefore, the existence of $\delta$ implies that the boundary components intersected by $\alpha$ are $P$-adjacent, and thus $[\alpha] \in Q_3$. This completes the proof of the claim.

\begin{figure}
  \labellist
  \small\hair 2pt
  \pinlabel $\XX{n}{b}$ at 88 70
  \pinlabel $\XX{n}{b}$ at 320 70
  \pinlabel $\alpha$ at 88 240
  \endlabellist
  \centering
  \includegraphics[width=0.6\textwidth]{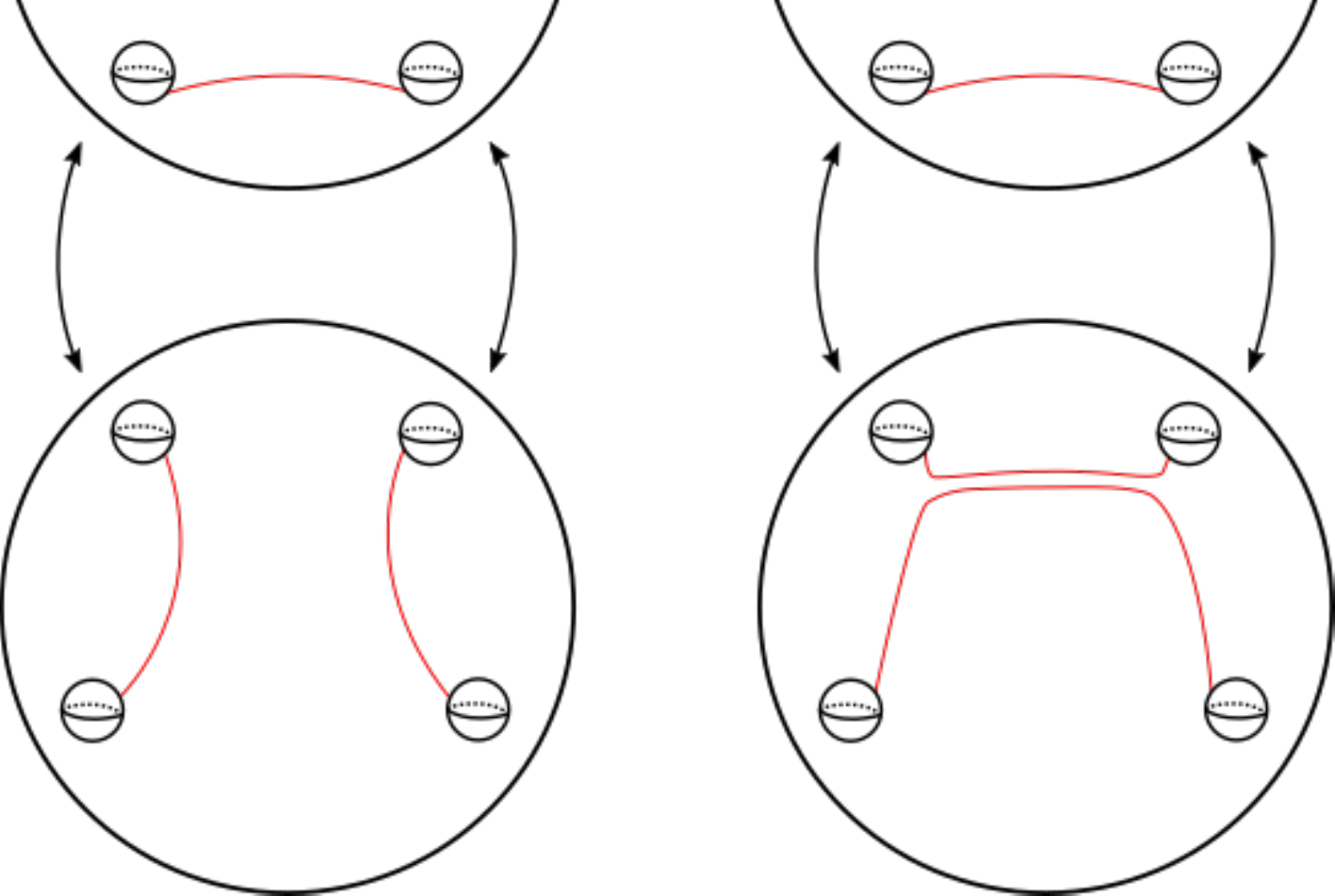}
  \caption{A loop can be surgered into a collection of loops which intersect $\partial \XX{n}{b}$ exactly twice.}
  \label{surger}
  \end{figure}

Let $f \in \IOpar{n}{b}{P}$. By the definition of $\IOpar{n}{b}{P}$, the element $\iota_*(f)$ acts trivially on $Q_2$. Moreover, $\iota_*(f)$ acts trivially on $Q_1$ by the definition of $\iota_*$. Lastly, suppose that $[h_1 \cdot h_2] \in Q_3$. Then $\iota_*(f)$ fixes the homology class of $h_1$ since $f \in \IOpar{n}{b}{P}$, and fixes $h_2$ pointwise by the definition of $\iota_*$. Therefore, $f \in \iota_*^{-1}(\IO_m)$. 

Next, suppose that $f \in \iota_*^{-1}(\IO_m)$. By definition, $\iota_*(f)$ acts trivially on $H_1(\XXnobound{m})$, and thus on $Q_2$ as well since the map $H_1(\XX{n}{b}) \to H_1(\XXnobound{m})$ induced by $\iota$ is injective. This implies that $f$ acts trivially on homology classes of simple closed curves in $\XX{n}{b}$. So, we only need to check that $f$ preserves the homology classes of arcs in $\XXnobound{m}$ connecting distinct $P$-adjacent boundary components. Suppose there is a class of arcs $[\alpha] \in H_1^P(\XX{n}{b})$. Since $P$ is the partition of the boundary components induced by $\iota$, $[\alpha]$ can be completed to a homology class $[\alpha \cdot \beta] \in H_1(\XXnobound{m})$, where $\beta$ is an arc in $\XXnobound{m} \setminus \XX{n}{b}$ connecting the endpoints of $\alpha$. Then since $\iota_*(f) \in \IO_m$ and $\iota_*(f)$ fixes $\beta$ pointwise, we have
\begin{equation*}
  0 = ([\alpha \cdot \beta]) - \iota_*(f)([\alpha \cdot \beta]) = [\alpha] - f([\alpha]).
\end{equation*}
This shows that $f$ acts trivially on $[\alpha]$. Therefore, $f \in \IOpar{n}{b}{P}$. 
\end{proof}

\section{Birman exact sequence}\label{birmansection}

In this section, we give a version of the Birman exact sequence for the groups $\IOpar{n}{b}{P}$. We will start by giving a Birman exact sequence on the level of mapping class groups. We note that Banks has proved a version of the Birman exact sequence for 3-manifolds (see \cite{birmanexact3mfld}). However, this version involves forgetting a puncture rather than capping a boundary component, so we will prove our own version here. Once we have the sequence for mapping class groups, we will mod out by sphere twists to get a corresponding sequence for $\Outb{n}{b}$, and finally restrict to get a sequence for $\IOpar{n}{b}{P}$.

\begin{remark}
  In the following theorems, we exclude the case $(n, b) = (1,1)$. This is because boundary drags in $\Mod(\XX{1}{1})$ are trivial (see the proof of Theorem \ref{birmanmcg} for the definition of boundary drags). In this case, we have isomorphisms
  \begin{itemize}
    \item $\Mod(\XX{1}{1}) \cong \Mod(\XXnobound{1}) \cong \mathbb{Z}/2 \times \mathbb{Z}/2$,
    \item $\Outb{1}{1} \cong \Out(F_1) \cong \mathbb{Z}/2$,
    \item $\IOpar{1}{1}{\{\partial\}} \cong \IO_1 \cong 1$,
  \end{itemize}
  where one of the generators of $\Mod(\XXnobound{1}) = \Mod( S^1 \times S^2 )$ is a sphere twist about the sphere $* \times S^2$ and the other is the antipodal map in both coordinates. 

\end{remark}

\begin{theorem}\label{birmanmcg}
  Fix $n, b > 0$ such that $(n,b) \neq (1,1)$. Glue a ball to a boundary component of $\XX{n}{b}$, and let $\XX{n}{b} \hookrightarrow \XX{n}{b-1}$ be the resulting embedding. Fix $x \in \XX{n}{b-1} \setminus \XX{n}{b}$.
\begin{enumerate}[(a)]
  \item If $b > 1$, choose a lift $\tilde{x} \in \Fr_x(\XX{n}{b-1})$ of $x$, where $\Fr(\XX{n}{b-1})$ is the oriented frame bundle of $\XX{n}{b-1}$. We then have an exact sequence
    \begin{equation*}
      1 \to \pi_1(\Fr(\XX{n}{b-1}), \tilde{x}) \to \Mod(\XX{n}{b}) \to \Mod(\XX{n}{b-1}) \to 1.
    \end{equation*}

  \item If $b = 1$ and $n > 1$, then we have an exact sequence
  \begin{equation*}
    1 \to \pi_1(\XX{n}{b-1}, x) \to \Mod(\XX{n}{b}) \to \Mod(\XX{n}{b-1}) \to 1.
  \end{equation*}
\end{enumerate}
\end{theorem}

\begin{proof}
 Let $\Diff(\XX{n}{b-1})$ denote the space of orientation-preserving diffeomorphisms of $\XX{n}{b-1}$ which restrict to the identity on $\partial \XX{n}{b-1}$, and let $\Diff(\XX{n}{b-1}, \tilde{x})$ be the subspace of $\Diff(\XX{n}{b-1})$ consisting of diffeomorphisms which fix the framing $\tilde{x}$. It is standard that there is a fiber bundle
\begin{equation}\label{fibbundle}
  \Diff(\XX{n}{b-1}, \tilde{x}) \to \Diff(\XX{n}{b-1}) \overset{p}{\to} \Fr(\XX{n}{b-1}),
\end{equation}
where the map $p: \Diff(\XX{n}{b-1}) \to \Fr(\XX{n}{b-1})$ is given by $\varphi \mapsto d\varphi(\tilde{x})$. 
Passing to the long exact sequence of homotopy group associated to this fiber bundle, we find the segment
\begin{equation*}
\pi_1(\Fr(\XX{n}{b-1})) \to \pi_0(\Diff(\XX{n}{b-1}, \tilde{x})) \to \pi_0(\Diff(\XX{n}{b-1})) \to \pi_0(\Fr(\XX{n}{b-1})).
\end{equation*}
Since $\Fr(\XX{n}{b-1})$ is the \emph{oriented} frame bundle, it is connected, and so $\pi_0(\Fr(\XX{n}{b-1}))$ is trivial. Moreover, $\pi_0(\Diff(\XX{n}{b-1}, \tilde{x}))$ is isomorphic to $\Mod(\XX{n}{b})$. For a proof of this fact in the surface case, see \cite[p.~102]{primer}; the proof goes exactly the same way in our setting. Therefore, the above sequence becomes
\begin{equation} \label{partialbirman}
  \pi_1(\Fr(\XX{n}{b-1})) \to \Mod(\XX{n}{b}) \to \Mod(\XX{n}{b-1}) \to 1.
\end{equation}
To get a short exact sequence, we must understand the kernel of the map $\PushFrame: \pi_1(\Fr(\XX{n}{b-1})) \to \Mod(\XX{n}{b})$. We remark here that the map $\PushFrame$ is given by pushing and rotating a small ball containing $x$ about a loop based at $x$. This is in analogy with the ``disk pushing maps" seen in the case of surfaces. Since $\XX{n}{b-1}$ is parallelizable, we have
\begin{equation*}
\pi_1(\Fr(\XX{n}{b-1})) \cong \pi_1(\XX{n}{b-1}) \times \pi_1(SO(3)) = \pi_1(\XX{n}{b-1}) \times \mathbb{Z}/2.
\end{equation*}
Consider the map $\Mod(\XX{n}{b}) \to \Aut(\pi_1(\XX{n}{b}, y))$, where the basepoint $y$ is on the boundary component $\partial$ being capped off. As is shown in Figure \ref{conjugate}, the composition
\begin{equation*}
\pi_1(\Fr(\XX{n}{b-1})) \cong \pi_1(\XX{n}{b-1}) \times \mathbb{Z}/2 \overset{\PushFrame}{\longrightarrow} \Mod(\XX{n}{b}) \to \Aut(\pi_1(\XX{n}{b}, y))
\end{equation*}
is given by conjugation about the loop being pushed around. Since $\Aut(\pi_1(\XX{n}{b}, y)) \cong \Aut(F_n)$ is centerless for $n > 1$, the entire kernel of $\PushFrame$ must be contained in $1 \times \mathbb{Z}/2 \subset \pi_1(\XX{n}{b-1}) \times \mathbb{Z}/2$. However, the image of the generator of this subgroup in $\Mod(\XX{n}{b})$ is the sphere twist $T_\partial$. By Theorems \ref{boundaryprodtrivial} and \ref{boundarytwistnontriv}, this sphere twist is nontrivial if and only if $b > 1$. If $b > 1$, this shows that $\Push$ is injective, and (\ref{partialbirman}) gives us the desired exact sequence. On the other hand, if $b=1$, then $\ker(\PushFrame) = 1 \times \mathbb{Z}/2$. Therefore, the image of $\PushFrame$ in $\Mod(\XX{n}{b})$ is isomorphic to 
\begin{equation*}
  \pi_1(\Fr(\XX{n}{b-1})) / \langle T_\partial \rangle \cong \pi_1(\XX{n}{b-1})
\end{equation*}
as desired. 
\end{proof}

\begin{figure}
\labellist
\small\hair 2pt
\pinlabel $\alpha$ at 150 165
\pinlabel $\gamma$ at 150 85
\pinlabel $\gamma^{-1}\alpha\gamma$ at 480 145
\pinlabel $y$ at 245 142
\endlabellist
\centering
\includegraphics[width=0.6\textwidth]{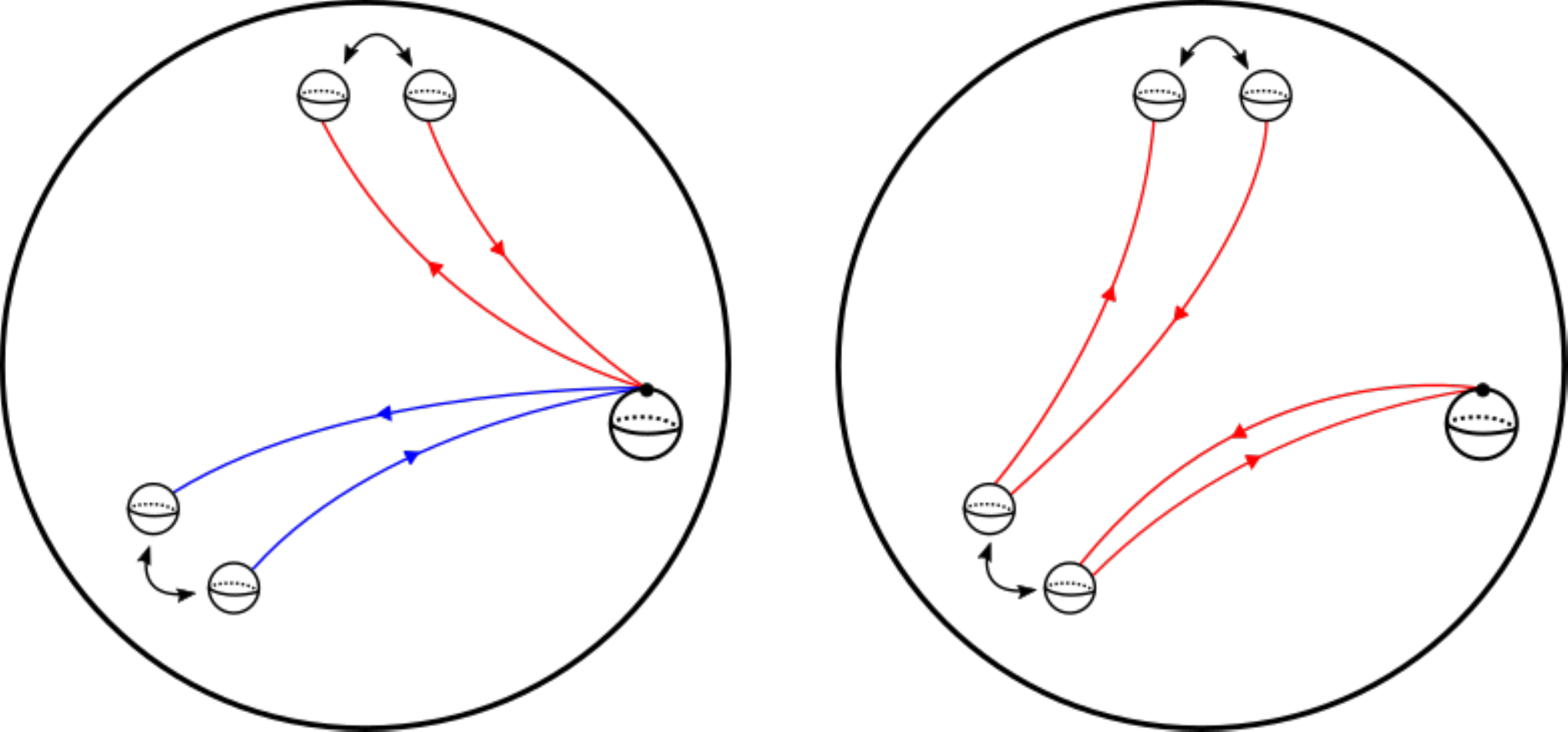}
\caption{The image of $\alpha$ under $\PushFrame(\gamma, T)$ is $\gamma^{-1}\alpha\gamma$. Here, $T$ can be either $T_\partial$ or trivial.}
\label{conjugate}
\end{figure}
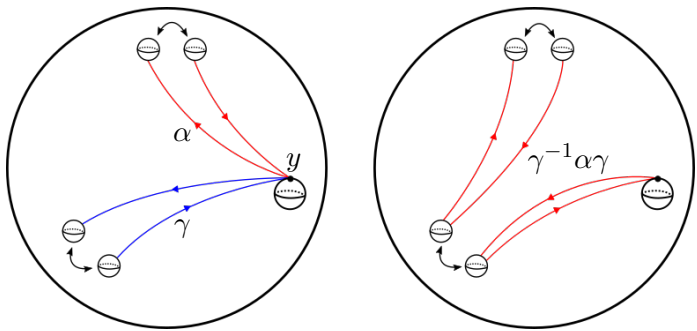



\paragraph{Modding out by sphere twists.} Now that we have a Birman exact sequence for $\Mod(\XX{n}{b})$, we can mod out by sphere twists to get a Birman exact sequence for $\Outb{n}{b}$. Consider the map $i_M: \Mod(\XX{n}{b}) \to \Mod(\XX{n}{b-1})$ given by capping off a boundary component $\partial$. Since $\iota_M$ takes sphere twists to sphere twists, this map descends to a map $\iota_*: \Outb{n}{b} \to \Outb{n}{b-1}$. Since $\iota_M$ is surjective, $\iota_*$ is as well. Let $K$ be the kernel of $\iota_*$, and let $\psi: \Mod(\XX{n}{b}) \to \Outb{n}{b}$ be the quotient map. If $b > 1$, then the kernel of $\iota_M$ is $\pi_1(\Fr(\XX{n}{b-1}), \widetilde{x})$ by Theorem \ref{birmanmcg}. Let $\PushFrame: \pi_1(\Fr(\XX{n}{b-1}), \widetilde{x}) \to \Mod(\XX{n}{b})$ be the map defined in the proof of Theorem \ref{birmanmcg}, and fix an identification $\pi_1(\Fr(\XX{n}{b-1}), \widetilde{x}) = \pi_1(\XX{n}{b-1}, x) \times \mathbb{Z}/2$. Since
\begin{equation*}
  \iota_*(\psi(\PushFrame(\gamma, T))) = \psi(\iota_M(\PushFrame(\gamma, T))) = \psi(\id) = \id
\end{equation*}
for all $(\gamma, T) \in \pi_1(\Fr(\XX{n}{b-1}), \widetilde{x})$, the image of $\pi_1(\Fr(\XX{n}{b-1}), \widetilde{x})$ under $\psi \circ \PushFrame$ is contained in $K$. In other words, we have the following commutative diagram:
\begin{center}
\begin{tikzcd}
1 \arrow[r] & \pi_1(\Fr(\XX{n}{b-1}), \tilde{x}) \arrow[r, "\PushFrame"] \arrow[d, "\psi_P"] & \Mod(\XX{n}{b}) \arrow[r, "\iota_M"] \arrow[d, "\psi"] & \Mod(\XX{n}{b-1}) \arrow[r] \arrow[d] & 1 \\
1 \arrow[r] & K \arrow[r] & \Outb{n}{b} \arrow[r, "\iota_*"] & \Outb{n}{b-1} \arrow[r] & 1,
\end{tikzcd}
\end{center}
where $\psi_P = \psi \circ \PushFrame$. Next, we claim that the map $\psi_P: \pi_1(\Fr(\XX{n}{b-1}), \widetilde{x}) \to K$ is surjective. To see this, let $f \in K$, and choose a lift $\mathfrak{f} \in \Mod(\XX{n}{b})$ of $f$. Since $\iota_*(f) = \id$, the image $\iota_M(\mathfrak{f})$ is a product of sphere twists $T_{S_1} \cdots T_{S_j}$. For each $T_{S_i} \in \Mod(\XX{n}{b-1})$, choose a preimage $T_{S_i}' \in \Mod(\XX{n}{b})$ which is also a sphere twist. Then
\begin{equation*}
  \iota_M(T_{S_1}' \cdots T_{S_j}' \cdot \mathfrak{f}) = \id,
\end{equation*}
which implies that $T_{S_1}' \cdots T_{S_j}' \cdot \mathfrak{f} = \PushFrame(\gamma, T)$ for some $(\gamma, T) \in \pi_1(\XX{n}{b-1}, x) \times \mathbb{Z}/2\mathbb{Z}$. Moreover, $\psi(T_{S_1}' \cdots T_{S_j}' \cdot \mathfrak{f}) = f$, which verifies our claim that $\psi_P: \pi_1(\Fr(\XX{n}{b-1}), \widetilde{x}) \to K$ is surjective. 

Now, we wish to identify the kernel of $\psi_P$. Let $(\gamma, T) \in \pi_1(\XX{n}{b-1}, \widetilde{x})$, and fix a basepoint $y$ on the boundary component being capped off. At the end of the proof of Theorem \ref{birmanmcg}, we saw that $\PushFrame(\gamma, T)$ acts nontrivially on $\pi_1(\XX{n}{b}, y)$ if and only if $\gamma$ is trivial. Since sphere twists act trivially on homotopy classes of curves, it follows that $\psi_P(\gamma, T)$ is nontrivial if $\gamma$ is nontrivial. Therefore, the kernel of $\psi_P$ must lie inside $1 \times \mathbb{Z}/2\mathbb{Z} \subset \pi_1(\XX{n}{b-1}, x) \times \mathbb{Z}/2\mathbb{Z}$. However, the generator of $1 \times \mathbb{Z}/2\mathbb{Z}$ gets mapped to $T_\partial$ under $\PushFrame$, which is killed in $\Outb{n}{b}$. Therefore, $\ker(\psi) = 1 \times \mathbb{Z}/2\mathbb{Z}$, and so it follows that  $K \cong \pi_1(\XX{n}{b-1}, x)$. 

On the other hand, if $b=1$ and $n > 1$, then the kernel of the map $\iota_M: \Mod(\XX{n}{b}) \to \Mod(\XX{n}{b-1})$ is $\pi_1(\XX{n}{b-1}, x)$ by Theorem \ref{birmanmcg}. Almost exactly the same argument used above shows that the quotient map restricts to a surjective map $\psi_P: \pi_1(\XX{n}{b-1}, x) \to K$. However, in this case, $\psi_P$ is injective since the sphere twist $T_\partial$ has already been killed off. Thus, we find that $K \cong \pi_1(\XX{n}{b-1}, x)$ in this case as well.

From now on, we will identify the kernel of the map $\iota_*: \Outb{n}{b} \to \Outb{n}{b-1}$ with $\pi_1(\XX{n}{b-1}, x)$. The map $\pi_1(\XX{n}{b-1}, x) \to \Outb{n}{b}$ will play a significant role throughout the remainder of the paper, and so we give a formal definition here.

\begin{definition}
  The map $\Push: \pi_1(\XX{n}{b-1}, x) \to \Outb{n}{b}$ is defined as $\Push(\gamma) = \psi(\PushFrame(\gamma, T))$, where $T \in \mathbb{Z}/2\mathbb{Z}$ is arbitrary. Since sphere twists become trivial in $\Outb{n}{b}$, this element depends only on $\gamma$. 
\end{definition}

The upshot of this is that we have proven the Birman exact sequence for $\Outb{n}{b}$.

\begin{theorem}\label{birmanout}
Fix $n, b > 0$ such that $(n,b) \neq (1,1)$, and let $\XX{n}{b} \hookrightarrow \XX{n}{b-1}$ be an embedding obtained by gluing a ball to a boundary component. Fix $x \in \XX{n}{b-1} \setminus \XX{n}{b}$. Then the following sequence is exact:
\begin{equation*}
1 \to \pi_1(\XX{n}{b-1}, x) \overset{\Push}{\to} \Outb{n}{b} \overset{\iota_*}{\to} \Outb{n}{b-1} \to 1.
\end{equation*}
\end{theorem}

\paragraph{Restrict to Torelli.} We now move on to proving Theorem \ref{birman}, which gives a Birman exact sequence for $\IOpar{n}{b}{P}$. We start by recalling its statement. Let $P$ be a partition of the boundary components of $\XX{n}{b}$, and fix a boundary component $\partial$. Let $p \in P$ be the set containing $\partial$, and let $\iota: \XX{n}{b} \hookrightarrow \XX{n}{b-1}$ be the inclusion obtained by capping off $\partial$. The partition $P$ induces a partition $P'$ of the boundary components of $\XX{n}{b-1}$ by removing $\partial$ from $p$. With this definition of $P'$, the map $\iota_*: \Outb{n}{b} \to \Outb{n}{b-1}$ restricts to a map $\IOpar{n}{b}{P} \to \IOpar{n}{b-1}{P'}$, which we will also call $\iota_*$. The sequence from Theorem \ref{birmanout} then restricts to
\begin{equation*}
  1 \to \pi_1(\XX{n}{b-1}) \cap \IOpar{n}{b}{P} \to \IOpar{n}{b}{P} \overset{\iota_*}{\to} \IOpar{n}{b-1}{P'}.
\end{equation*}
Theorem \ref{birman} asserts that $\iota_*$ is surjective, and identifies its kernel $\pi_1(\XX{n}{b-1}) \cap \IOpar{n}{b}{P}$. We start with surjectivity.

\begin{lemma}\label{torellisurjective}
The induced map $\iota_*: \IOpar{n}{b}{P} \to \IOpar{n}{b-1}{P'}$ is surjective for any embedding $\iota : \XX{n}{b} \hookrightarrow \XX{n}{b-1}$.
\end{lemma}

\begin{proof}
  Consider an element $g \in \IOpar{n}{b-1}{P'}$. Our goal is to find some $f \in \IOpar{n}{b}{P}$ such that $\iota_*(f) = g$. There are two cases.

  First, suppose that $p = \{ \partial \}$. Then the inclusion $\iota$ induces an isomorphism $\iota_H: H_1^P(\XX{n}{b}) \to H_1^{P'}(\XX{n}{b-1})$ which is equivariant with respect to the actions of $\Outb{n}{b}$ and $\Outb{n}{b-1}$. In other words, for any $[h] \in H_1^P(\XX{n}{b})$ and $f \in \Outb{n}{b}$, we have 
  \begin{equation}\label{equivariance}
    \iota_H(f \cdot [h]) = \iota_*(f) \cdot \iota_H([h]).
  \end{equation} 
  By Theorem \ref{birmanout}, there exists some $f \in \Outb{n}{b}$ such that $\iota_*(f) = g$. We claim that $f \in \IOpar{n}{b}{P}$. To see this, let $[h] \in H_1^P(\XX{n}{b})$. Then, by Equation (\ref{equivariance}), we see that 
  \begin{equation*}
    \iota_H(f \cdot [h]) = \iota_*(f) \cdot \iota_H([h]) = g \cdot \iota_H([h]) = \iota_H([h]).
  \end{equation*}
  Since $\iota_H$ is an isomorphism, this implies that $f \cdot [h] = [h]$, and so $f \in \IOpar{n}{b}{P}$, as desired.

  Next, suppose that $p \neq \{ \partial \}$. Again, choose $f \in \Outb{n}{b}$ such that $\iota_*(f) = g$. In this case, there is no longer a well-defined map $H_1^P(\XX{n}{b}) \to H_1^{P'}(\XX{n}{b-1})$. However, there is a subgroup of $H_1^P(\XX{n}{b})$ which projects isomorphically onto $H_1^{P'}(\XX{n}{b-1})$. Let $A \subset H_1^P(\XX{n}{b})$ be the subgroup generated by 
  \begin{align*}
    \{  [a] \in H_1(\XX{n}{b}, \partial\XX{n}{b}) \mid &\text{ either $a$ is a simple closed curve} \\
    &\text{ or $a$ is a properly embedded arc with} \\
    &\text{ neither endpoint on $\partial$}  \}.
  \end{align*}
  It is clear that $A \cong H_1^{P'}(\XX{n}{b-1})$. 
  
  Let $[k] \in H_1^P(\XX{n}{b})$ be the class of an arc $k$ which has an endpoint on $\partial$. We claim that $H_1^P(\XX{n}{b})$ is generated by $A$ and $[k]$. To establish this claim, it suffices to show that $[\ell] \in \langle A, [k] \rangle$, where $[\ell] \in H_1^P(\XX{n}{b})$ is the class of any arc with an endpoint on $\partial$ and the other elsewhere. Such an $\ell$ exists since $p \neq \{\partial\}$. Fix such a class $[\ell]$, and let $\alpha \subset \partial$ be an arc connecting the endpoints of $\ell$ and $k$ on $\partial$. Orient $\ell$, $\alpha$, and $k$ such that the curve $\ell \cdot \alpha \cdot k$ is well-defined. 

  If the endpoints of $\ell$ and $k$ which are not on $\partial$ lie on distinct boundary components, then $\ell \cdot \alpha \cdot k$ is an arc connecting $P'$-adjacent boundary components. Therefore, $[\ell] + [\alpha] + [k] \in A$. Since $[\alpha] = 0$ in $H_1^P(\XX{n}{b})$, it follows that $[\ell] \in \langle A, [k] \rangle$. On the other hand, if the endpoints of $\ell$ and $k$ which are not on $\partial$ lie on the same boundary component $\partial'$, then we can complete $\ell \cdot \alpha \cdot k$ to a loop $\ell \cdot \alpha \cdot k \cdot \beta$, where $\beta \subset \partial'$ is an arc connecting the endpoints of $\ell$ and $k$. Then
  \begin{equation*}
    [\ell] + [k] = [\ell] + [\alpha] + [k] + [\beta] = [\ell \cdot \alpha \cdot k \cdot \beta]\in A,
  \end{equation*}
  and so $[\ell] \in \langle A, [k] \rangle$. This completes the proof of the claim that $H_1^P(\XX{n}{b})$ is generated by $A$ and $[k]$.
  
  Since $A$ projects isomorphically onto $H_1^{P'}(\XX{n}{b-1})$, and this projection is equivariant with respect to the actions of $\Outb{n}{b}$ and $\Outb{n}{b-1}$, we have $f \cdot [a] = [a]$. It follows that $f$ acts trivially on $A$. Therefore, if $f$ fixes $[k]$, then $f \in \IOpar{n}{b}{P}$ by the discussion in the preceding paragraph, and so we are done. On the other hand, if $f$ does not fix $[k]$, then $\gamma = k \cdot f(k)^{-1}$ is a nontrivial loop based at a point on $\partial$. So, the element $\Push(\gamma)^{-1} \cdot f \in \Outb{n}{b}$ fixes $[k]$. Moreover, $\Push(\gamma)$ acts trivially on $A$, and so $\Push(\gamma)^{-1} \cdot f$ does as well. Thus, $\Push(\gamma)^{-1} \cdot f \in \IOpar{n}{b}{P}$. Finally, since $\Push(\gamma) \in \ker(i_*)$, we have that $\iota_*(\Push(\gamma)^{-1} \cdot f) = g$, and so we are done. 
\end{proof}

We now move on to the proof of Theorem \ref{birman}.

\begin{proof}[Proof of Theorem \ref{birman}]
Recall that we want to show that we have an exact sequence
\begin{equation*}
  1 \to L \overset{\Push}{\to} \IOpar{n}{b}{P} \overset{\iota_*}{\to} \IOpar{n}{b-1}{P'} \to 1,
\end{equation*}
where $L$ is equal to:
\begin{enumerate}[(a)]
  \item $\pi_1(\XX{n}{b-1}, x) \cong F_n$ if $p = \{\partial\}$.

  \item $[\pi_1(\XX{n}{b-1}, x), \pi_1(\XX{n}{b-1}, x)] \cong [F_n, F_n]$ if $p \neq \{ \partial \}$.
\end{enumerate}

By Lemma \ref{torellisurjective} and the discussion preceding it, all that is left to show is that $\pi_1(\XX{n}{b-1}) \cap \IOpar{n}{b}{P}$ agrees with the subgroups $L$ given above. 

We begin with the case $p = \{ \partial \}$. Recall that $\pi_1(\XX{n}{b-1})$ acts on $\XX{n}{b}$ by pushing the boundary component $\partial$ about a given loop. Since $\partial$ is not $P$-adjacent to any other boundary components, it follows that $\pi_1(\XX{n}{b-1})$ acts trivially on $H_1^P(\XX{n}{b})$. Therefore, $\pi_1(\XX{n}{b-1}) \subset \IOpar{n}{b}{P}$, and so $\pi_1(\XX{n}{b-1}) \cap \IOpar{n}{b}{P} = \pi_1(\XX{n}{b-1})$. This completes this case. 

Next, suppose that $p \neq \{ \partial \}$. In this case, not all elements of $\pi_1(\XX{n}{b})$ are contained in $\IOpar{n}{b}{P}$. This is because dragging $\partial$ about loops may change the homology class of arcs connected to $\partial$. In particular, if $\gamma \in \pi_1(\XX{n}{b-1})$ and $[h] \in H_1^P(\XX{n}{b})$ is the class of arc with an endpoint in $\partial$ and the other elsewhere, then $\Push(\gamma)$ acts on $[h]$ via
\begin{equation*}
\Push(\gamma) \cdot [h] = [\gamma] + [h].
\end{equation*}
See Figure \ref{haction} for an illustration. This implies that an element $\Push(\gamma)$ is in $\IOpar{n}{b}{P}$ if and only if $[\gamma] = 0$ in $H_1(\XX{n}{b-1})$. Thus,
\begin{equation*}
  \pi_1(\XX{n}{b-1}) \cap \IOpar{n}{b}{P} = [\pi_1(\XX{n}{b-1}), \pi_1(\XX{n}{b-1})],
\end{equation*}
which is what we wanted to show. 
\end{proof}

\begin{figure}
  \labellist
  \small\hair 2pt
  \pinlabel $h$ at 205 150
  \pinlabel $\alpha$ at 145 190
  \pinlabel $\partial$ at 235 170
  \endlabellist
  \centering
  \includegraphics[width=0.8\textwidth]{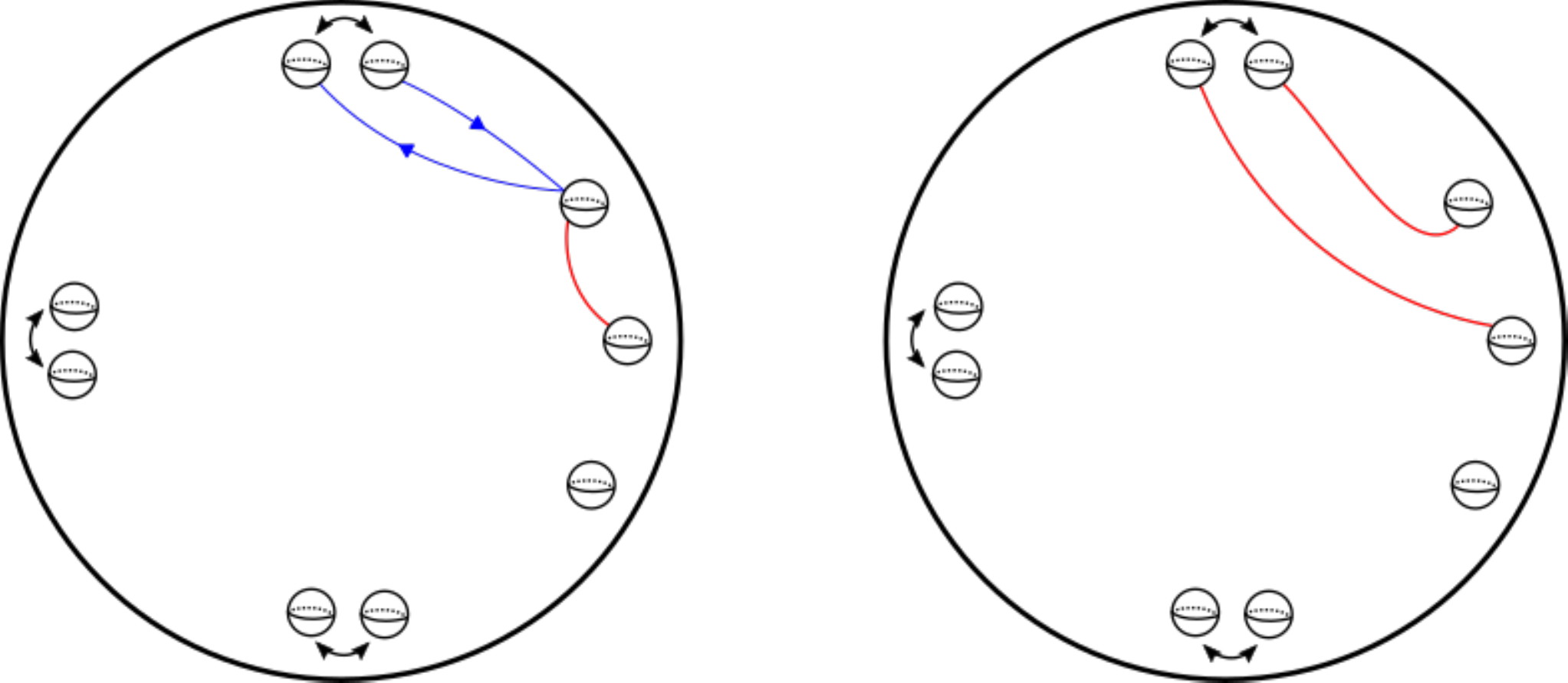}
  \caption{Dragging $\partial$ around $\alpha$ takes $[h]$ to $[\alpha] + [h]$.}
  \label{haction}
\end{figure}

\section{Generators}\label{generatorssection}

In this section, we will define our generators of $\IOpar{n}{b}{P}$. The definition of these generators will involve splitting and dragging boundary components, so we will discuss these processes in more detail first, then move on to the definitions. 

\paragraph{Splitting along spheres.} Let $S \subset \XX{n}{b}$ be an embedded 2-sphere. By \emph{splitting along $S$}, we mean removing an open tubular neighborhood $N$ of $S$ from $\XX{n}{b}$. If $S$ is nonseparating, the resulting manifold will be diffeomorphic to $\XX{n-1}{b+2}$ and if $S$ is separating, the result will be diffeomorphic to $\XX{m_1}{c_1} \sqcup \XX{m_2}{c_2}$, where $m_1 + m_2 = n$ and $c_1 + c_2 = b + 2$. Notice that the resulting manifold is a submanifold of $\XX{n}{b}$, and so we get a corresponding map $\Mod(\XX{n-1}{b+2}) \to \Mod(\XX{n}{b})$ if $S$ is nonseparating, or $\Mod(\XX{m_1}{c_1}) \times \Mod(\XX{m_2}{c_2}) \to \Mod(\XX{n}{b})$ if $S$ is separating. In either case, this map sends sphere twists to sphere twists, and thus induces a map $\iota_*: \Outb{n-1}{b+2} \to \Outb{n}{b}$ or $\iota_*: \Outb{m_1}{c_1} \times \Outb{m_2}{c_2} \to \Outb{n}{b}$, depending on whether or not $S$ separates $\XX{n}{b}$. 

\paragraph{Dragging boundary components.} Let $\partial$ be a boundary component of $\XX{n}{b}$, and let $\iota: \XX{n}{b} \hookrightarrow \XX{n}{b-1}$ be the embedding obtained by capping off $\partial$. By Theorem \ref{birmanout}, we have an exact sequence 
\begin{equation*}
  1 \to \pi_1(\XX{n}{b-1}, x) \overset{\Push}{\longrightarrow} \Outb{n}{b} \overset{\iota_*}{\longrightarrow} \Outb{n}{b-1} \to 1,
\end{equation*}
where $x \in \XX{n}{b-1} \setminus \XX{n}{b}$. Given $\gamma \in \pi_1(\XX{n}{b-1}, x)$, recall that the element $\Push(\gamma) \in \Out(\XX{n}{b})$ is given by pushing $\partial$ about the loop $\gamma$. In the remainder of this section, we will be dragging multiple boundary components at a time. So, from now on we will write $\Push_\partial(\gamma)$ in order to keep track of which boundary component is being pushed.

\paragraph{Magnus generators.} We now move on to defining our generators for $\IOpar{n}{b}{P}$. In the $b=0$ case, we have that $\IOpar{n}{0}{P} = \IO_n$, where $\IO_n$ is the subgroup of $\Out(F_n)$ acting trivially on homology. In \cite{magnus}, Magnus found the following generating set for $\IO_n$.

\begin{theorem}[Magnus]
Let $F_n = \langle x_1, \ldots, x_n \rangle$. The group $IO_n$ is generated by the $\Out(F_n)$-classes of the automorphisms
\begin{equation*}
M_{ij}: x_i \mapsto x_jx_ix_j^{-1}, \qquad M_{ijk}: x_i \mapsto x_i[x_j,x_k],
\end{equation*}
for all distinct $i,j,k \in \{1, \ldots, n\}$ with $j < k$. Here, the automorphisms are understood to fix $x_\ell$ for $\ell \neq i$. 
\end{theorem}

\begin{figure}
  \labellist
  \small\hair 2pt
  \pinlabel $x$ at 332 280
  \pinlabel $\alpha_1$ at 262 385
  \pinlabel $\alpha_2$ at 175 271
  \pinlabel $\alpha_3$ at 262 175
  \pinlabel $\sigma_1^+$ at 140 420
  \pinlabel $\sigma_1^-$ at 311 481
  \pinlabel $\sigma_2^+$ at 66 345
  \pinlabel $\sigma_2^-$ at 66 184
  \pinlabel $\sigma_3^-$ at 122 87
  \pinlabel $\sigma_3^+$ at 296 49
  \pinlabel $\partial_1^1$ at 446 420
  \pinlabel $\partial_2^1$ at 490 324
  \pinlabel $\partial_1^2$ at 490 213
  \pinlabel $\partial_1^3$ at 437 131
  \endlabellist
  \centering
  \includegraphics[width=0.5\textwidth]{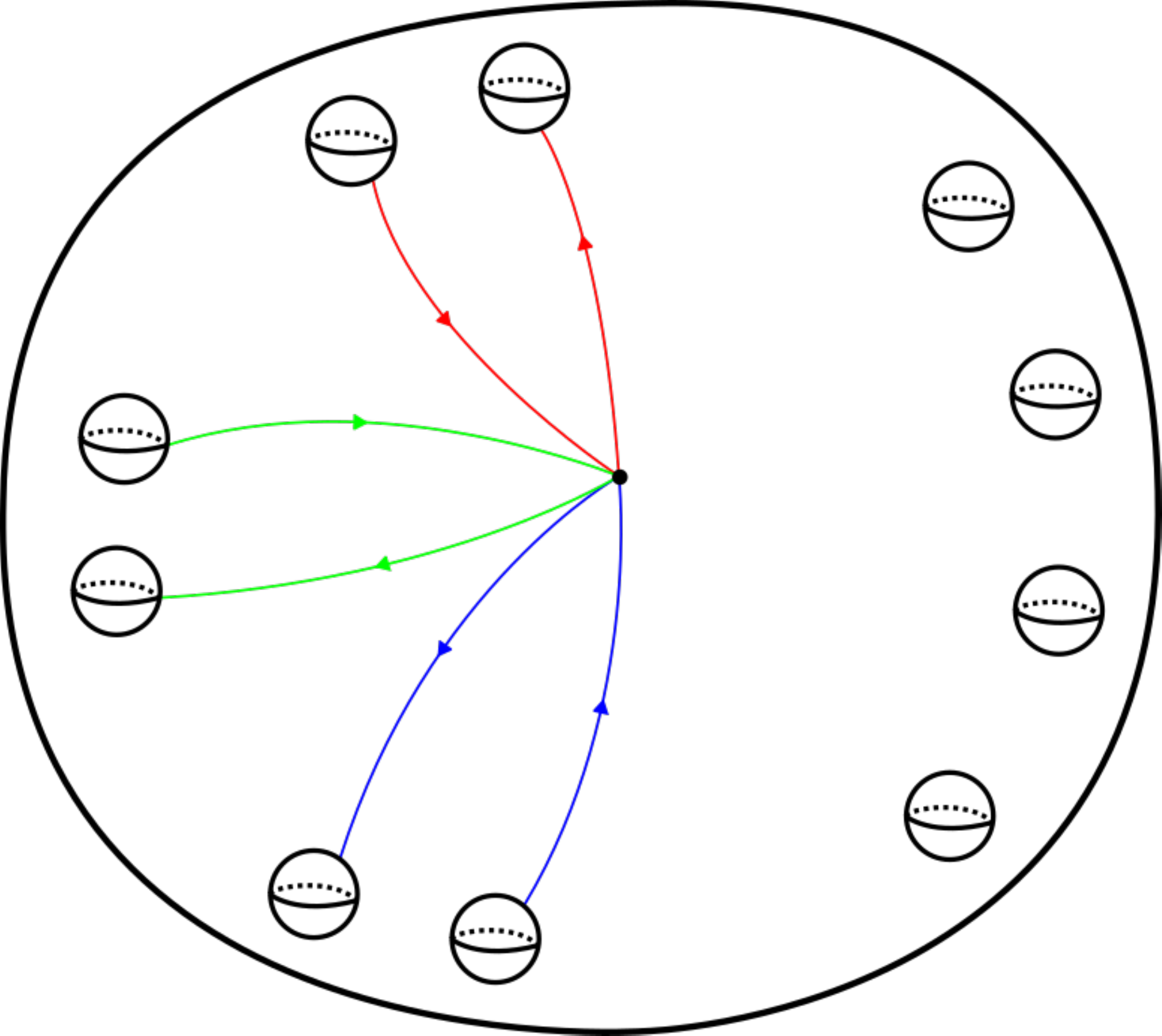}
  \caption{$\XX{3}{4}$ split along $S_1 \cup S_2 \cup S_3$ with the partition $P = \{\{\partial_1^1, \partial_2^1\}, \{\partial_1^2\}, \{\partial_1^3\}\}$.}
  \label{gensetup}
\end{figure}
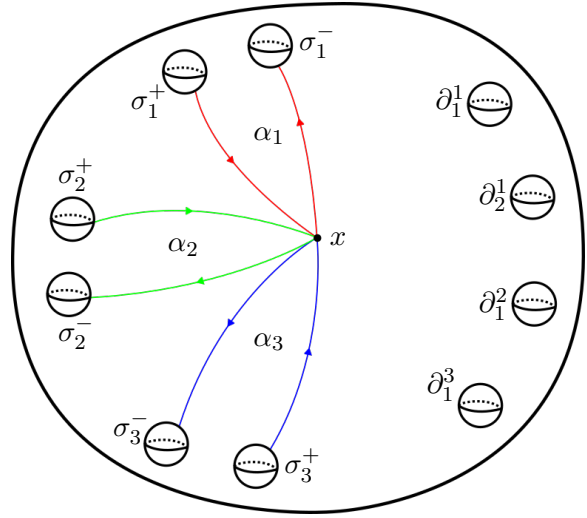

Our generating set will be inspired by Magnus's, and will indeed reduce to it when $b=0$. In order to choose a concrete collection of elements, we will need to make some choices. First, fix a basepoint $* \in \interior(\XX{n}{b})$ and a set $\{ \alpha_1, \ldots, \alpha_n \}$ of oriented simple closed curves intersecting only at $*$ whose homotopy classes form a free basis for $\pi_1(\XX{n}{b}, *)$. We will call such a set $\{\alpha_1, \ldots, \alpha_n\}$ a \emph{geometric free basis} for $\pi_1(\XX{n}{b},*)$. In addition, choose a corresponding \emph{sphere basis}; that is, a collection of $n$ disjointly embedded oriented 2-spheres $S_1, \ldots, S_n \subset \XX{n}{b}$ such that each $S_i$ intersects $\alpha_i$ exactly once with a positive orientation and is disjoint from the other $\alpha_j$. Notice that splitting $\XX{n}{b}$ along the $S_i$ reduces it to a 3-sphere $\mathcal{Z} \subset \XX{n}{b}$ with $b + 2n$ boundary components. The submanifold $\mathcal{Z}$ will play a significant role throughout the remainder of this section because it will allow all of our choices made in the definitions to be unique. For each $S_i$, let $\sigma_i^+$ and $\sigma_i^-$ be the boundary components of $\mathcal{Z}$ arising from the split along $S_i$, where $\sigma_i^+$ (resp. $\sigma_i^-$) is the component lying on the positive (resp. negative) side of $S_i$. We will also choose an ordering $P = \{p_1, \ldots, p_{\vert P \vert}\}$ and an ordering $p_r = \{\partial_1^r, \ldots, \partial_{b_r}^r\}$ for each $r \in \{1, \ldots, \vert P \vert\}$. See Figure \ref{gensetup}. 

The following lemma will be helpful in showing that our generators lie in $\IOpar{n}{b}{P}$. 

\begin{figure}
  \labellist
  \small\hair 2pt
  \pinlabel $h$ at 105 95
  \pinlabel $\alpha$ at 435 95
  \pinlabel $\gamma_s$ at 520 105
  \pinlabel $\gamma_e$ at 500 180
  \pinlabel $*$ at 462 141
  \endlabellist
  \centering
  \includegraphics[width=0.8\textwidth]{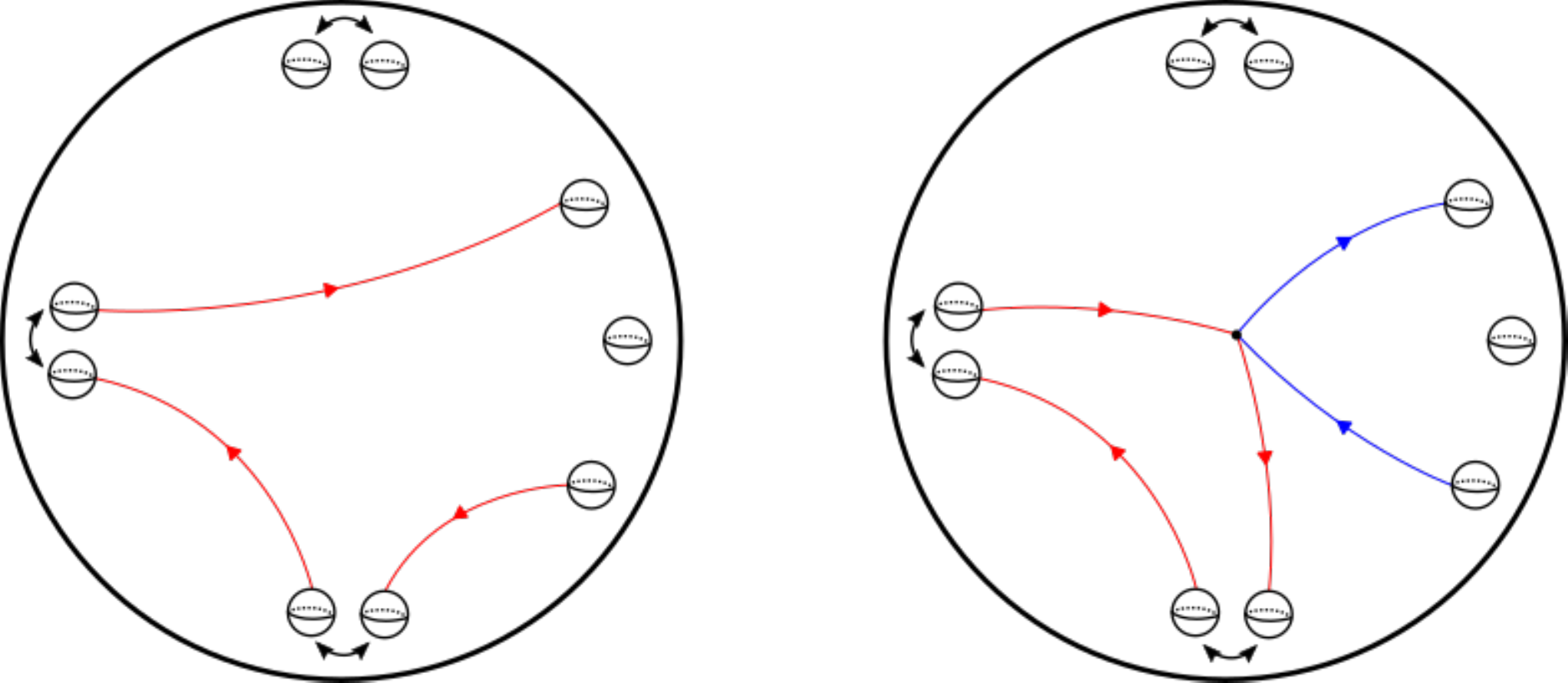}
  \caption{The arc $h$ homotoped to be put in the form $\gamma_s \cdot \alpha \cdot \gamma_e$.}
  \label{arcformfig}
\end{figure}
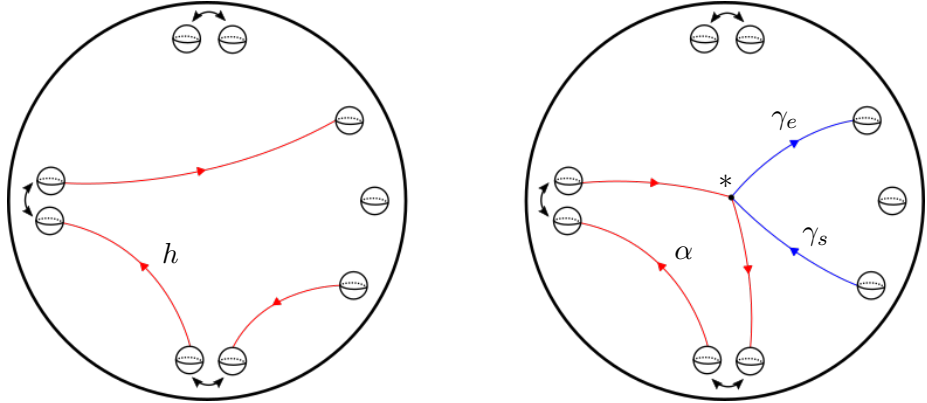

\begin{lemma}\label{arcform}
  Let $\mathcal{Z}$ be as above, and suppose that $h \subset \XX{n}{b}$ is a properly embedded oriented arc connecting $P$-adjacent boundary components of $\XX{n}{b}$. Then the homology class of $[h] \in H_1^P(\XX{n}{b})$ has the form 
  \begin{equation*}
    [h] = [\alpha] + [h_0],
  \end{equation*}
  where $\alpha$ is a loop based at $*$, and $h_0$ is the unique arc (up to isotopy) in $\mathcal{Z}$ which has the same endpoints as $h$. 
\end{lemma}

\begin{proof}
  We may homotope $h$ such that it has the form $h = \gamma_s \cdot \alpha \cdot \gamma_e$, where (see Figure \ref{arcformfig}):
  \begin{itemize}
    \item $\gamma_s \subset \mathcal{Z}$ is the unique arc (up to isotopy) from the initial point of $h$ to the basepoint $*$ of $\XX{n}{b}$,
    \item $\gamma_e \subset \mathcal{Z}$ is the unique arc from $*$ to the endpoint of $h$,
    \item $\alpha \in \pi_1(\XX{n}{b}, *)$. 
  \end{itemize}
  Then, 
  \begin{equation*}
    [h] = [\gamma_s \cdot \alpha \cdot \gamma_e] = [\alpha] + [\gamma_s \cdot \gamma_e] = [\alpha] + [h_0],
  \end{equation*}
  as desired. 
\end{proof}

  
\paragraph{Handle drags.} Let $i \in \{1, \ldots, n \}$, and let $h_i$ be the unique (up to isotopy) properly embedded arc in $\mathcal{Z}$ connecting $\sigma_i^+$ and $\sigma_i^-$ which is disjoint from the $\alpha_k$. Choose a tubular neighborhood $N_i$ of $\sigma_i^+ \cup h_i \cup \sigma_i^-$ that does not intersect any $\alpha_k$ for $k \neq i$. Let $\Sigma_i$ be the boundary component of $N_i$ which is not isotopic to $\sigma_i^+$ or $\sigma_i^-$ (notice that $\Sigma_i$ is diffeomorphic to a 2-sphere). Splitting $\XX{n}{b}$ along $\Sigma_i$ yields $\XX{n-1}{b+1} \sqcup \XX{1}{1}$. Let $\Sigma_i' \subset \partial\XX{n-1}{b+1}$ be the boundary component coming from this split, and fix a basepoint $y_i \in \Sigma_i'$. Fix an oriented arc $\delta_i \subset Z$ from $y_i$ to $*$ which only intersects $\Sigma_i'$ at $y_i$. Since $\mathcal{Z}$ is a 3-sphere with spherical boundary components, $\delta_i$ is unique up to isotopy. The arc $\delta_i$ induces an isomorphism $\pi_1(\XX{n-1}{b+1}, *) \to \pi_1(\XX{n-1}{b+1}, y_i)$ given by $\gamma \mapsto \delta_i\gamma\delta_i^{-1}$. Define $\beta_j^i = \delta_i\alpha_j\delta_i^{-1}$. Then we define the \emph{handle drag} $\HD_{ij} := \iota_*(\Push_{\Sigma_i'}(\beta_j^i), \id) \in \Outb{n}{b}$ for $j \neq i$, where $\iota_*$ is the map $\Outb{n-1}{b+1} \times \Outb{1}{1} \to \Outb{n}{b}$ induced by splitting along $\Sigma_i$. 

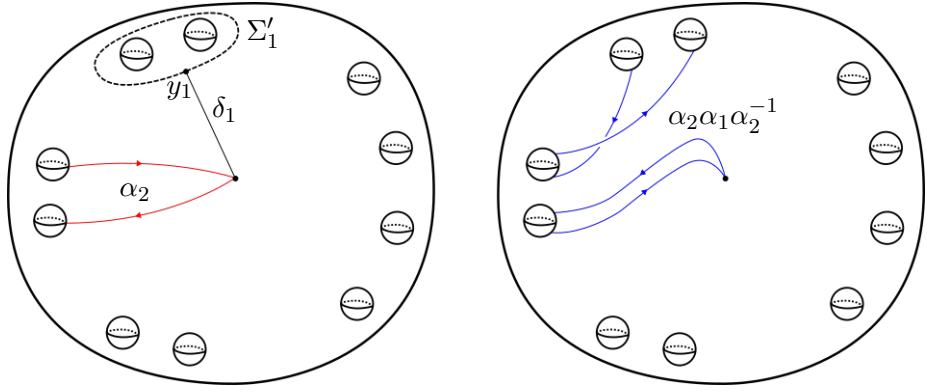
\begin{figure}
\labellist
\small\hair 2pt
\pinlabel $\Sigma_1'$ at 200 275
\pinlabel $\delta_1$ at 170 215
\pinlabel $y_1$ at 135 230
\pinlabel $\alpha_2$ at 100 150
\pinlabel $\alpha_2\alpha_1\alpha_2^{-1}$ at 560 210
\endlabellist
\centering
\includegraphics[width=0.8\textwidth]{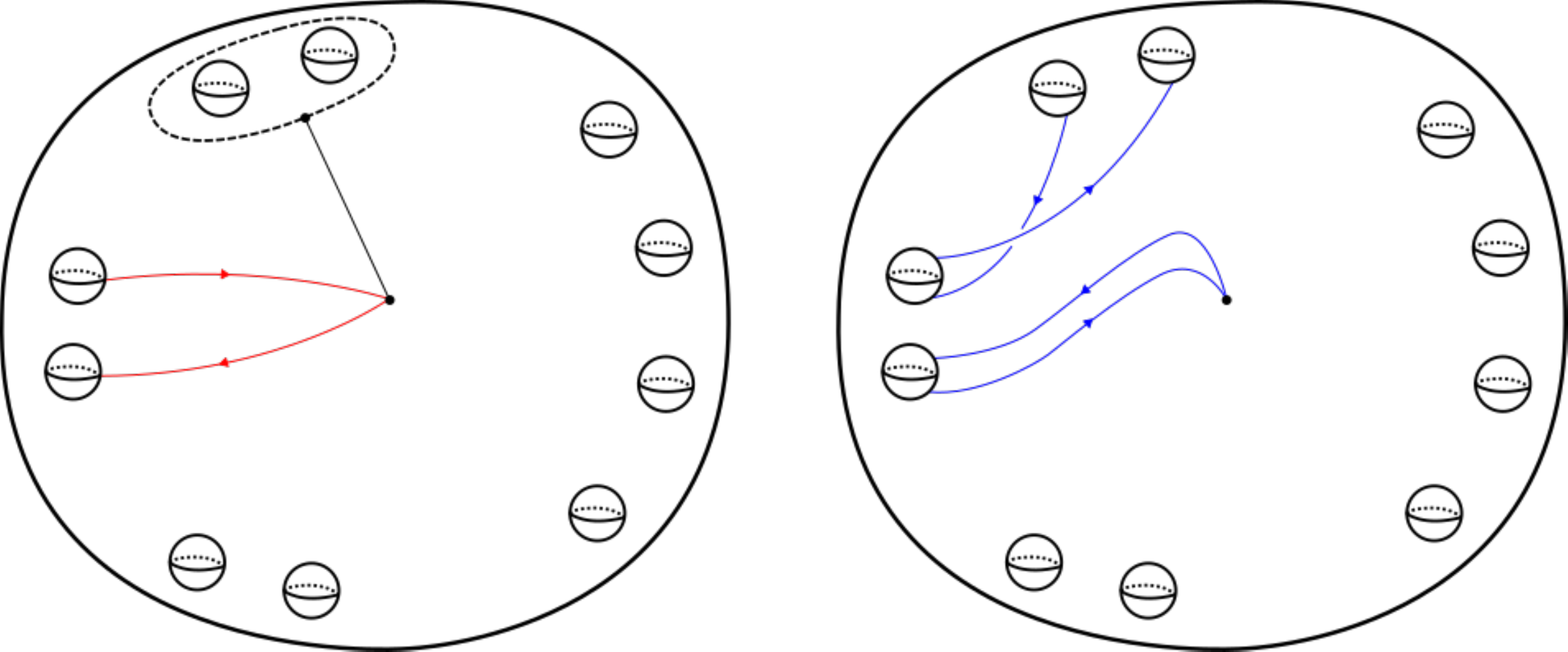}
\caption{Setup of the handle drag $\HD_{12}$ and the image of $\alpha_1$ under $\HD_{12}$}
\label{handledrag}
\end{figure}

To see that $\HD_{ij} \in \IOpar{n}{b}{P}$, notice that $\HD_{ij}$ acts trivially on $\alpha_k$ for $k \neq i$, and acts on $\alpha_i$ via $\alpha_i \mapsto \alpha_j\alpha_i\alpha_j^{-1}$. See Figure \ref{handledrag}. This shows that $\HD_{ij}$ acts trivially on homology classes of simple closed curves. Additionally, this shows that $\HD_{ij}$ reduces to $M_{ij}$ of the Magnus generators if $b=0$. 

Next, suppose that $h$ is an arc connecting $P$-adjacent boundary components. By Lemma \ref{arcform}, we may write $[h] = [\alpha] + [h_0]$, where $\alpha$ is a loop based at $*$, and $h_0$ is the unique arc (up to isotopy) in $\mathcal{Z}$ which has the same endpoints as $h$. We have seen that $\HD_{ij}$ fixes the homology class of $\alpha$. Moreover, we may homotope $\HD_{ij}$ such that it fixes the arc $h_0$. Thus, $\HD_{ij}$ fixes the homology class of $h$, and we conclude that $\HD_{ij} \in \IOpar{n}{b}{P}$. 

\paragraph{Commutator drags.} Let $i, j, k \in \{1, \ldots, n \}$ be distinct with $j < k$. Split $\XX{n}{b}$ along $S_i$ to get $\XX{n}{b+2}$, where $\mathcal{Z} \subset \XX{n}{b+2} \subset \XX{n}{b}$. Fix basepoint $y_i \in \sigma_i^+$ and $z_i \in \sigma_i^-$, and choose oriented arcs $\delta_i, \varepsilon_i \subset \mathcal{Z}$ connecting $y_i$ and $z_i$ to $*$, respectively. Just as in the construction of handle drags, $\delta_i$ and $\varepsilon_i$ are unique up to isotopy. Let $\beta_\ell^i = \delta_i\alpha_\ell\delta_i^{-1}$ and $\gamma_\ell^i = \varepsilon_i\alpha_\ell\varepsilon_i^{-1}$ for $\ell \in \{j, m\}$. Then, we  define the \emph{commutator drags} $\CD_{ijk}^+, \CD_{ijk}^- \in \Outb{n}{b}$ as $\iota_*(\Push_{\sigma_i^+}([\beta_j^i, \beta_k^i]))$ and $\iota_*(\Push_{\sigma_i^-}([\gamma_j^i, \gamma_k^i]))$, respectively, where $\iota_*: \Outb{n}{b+2} \to \Outb{n}{b}$ is the map induced by splitting along $S_i$. See Figure \ref{commdrag}.

Again, we see that $\CD_{ijk}^{\pm}$ acts trivially on $\alpha_\ell$ for $\ell \neq i$, the commutator drag $\CD_{ijk}^+$ sends $\alpha_i$ to $\alpha_i[\alpha_j, \alpha_k]^{-1}$, and $\CD_{ijk}^-$ sends $\alpha_i$ to $[\alpha_j,\alpha_k]\alpha_i$. This shows that $\CD_{ijk}^+$ reduces to $M_{ijk}^{-1}$ of the Magnus generators when $b=0$. 

Now, suppose that $h$ is an arc connecting $P$-adjacent boundary components of $\XX{n}{b}$. By Lemma \ref{arcform}, we may express $[h]$ in the form $[h] = [\alpha] + [h_0]$. We just saw that $\CD_{ijk}^{\pm}$ fixes $[\alpha]$. We may also homotope $\CD_{ijk}^{\pm}$ such that it fixes $h_0$. Thus, $\CD_{ijk}^{\pm}$ fixes $[h]$, and so $\CD_{ijk}^{\pm} \in \IOpar{n}{b}{P}$.

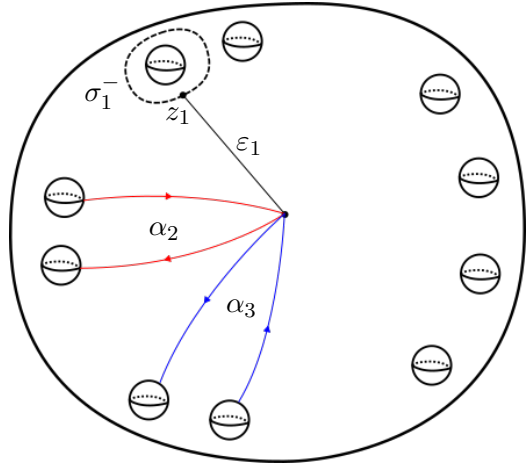
\begin{figure}
\labellist
\small\hair 2pt
\pinlabel $z_1$ at 109 225
\pinlabel $\varepsilon_1$ at 155 205
\pinlabel $\alpha_2$ at 100 150
\pinlabel $\alpha_3$ at 150 100
\pinlabel $\sigma_1^-$ at 60 240
\endlabellist
\centering
\includegraphics[width=0.45\textwidth]{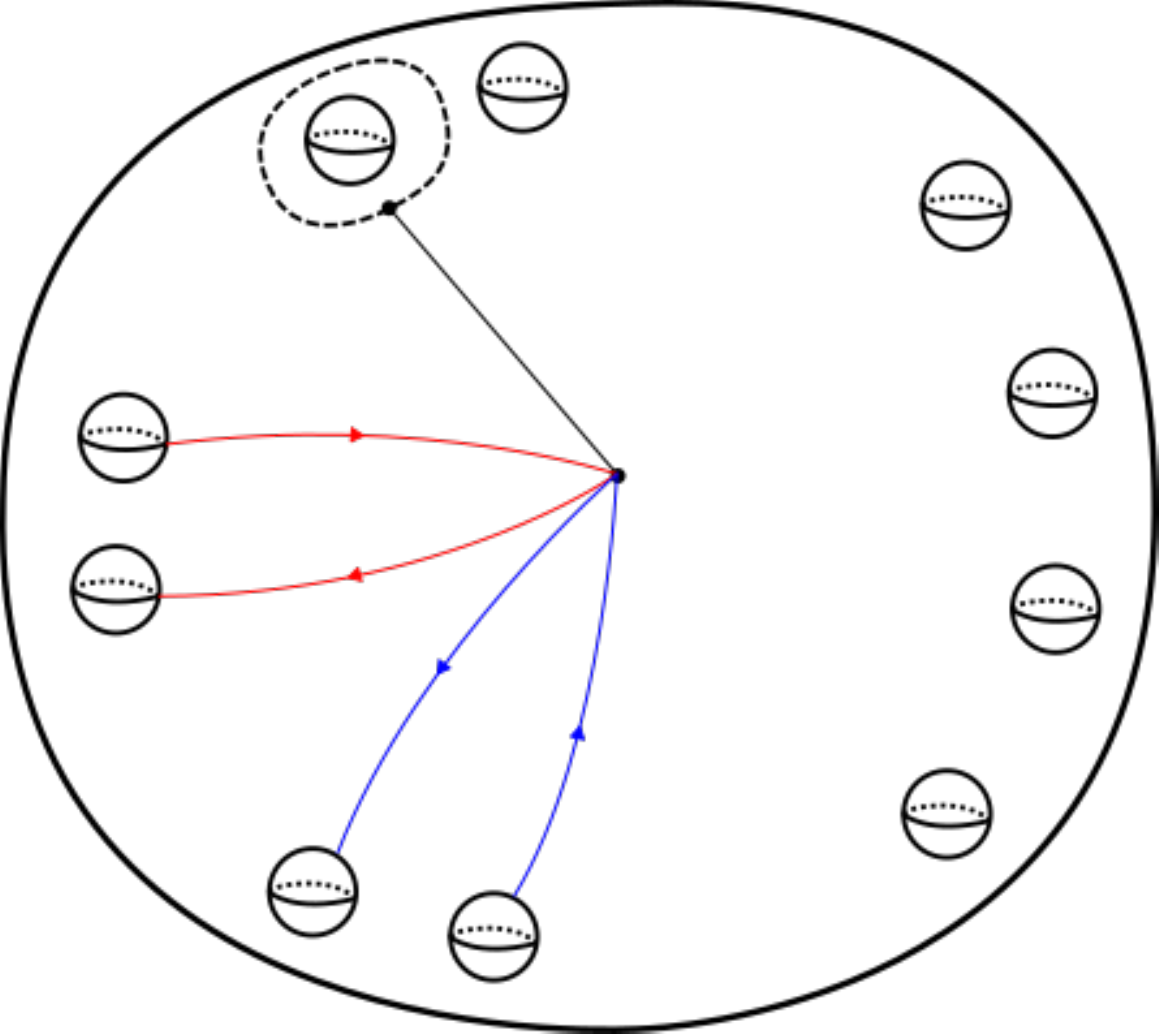}
\caption{Setup of the commutator drag $\CD_{123}^-$.}
\label{commdrag}
\end{figure}

\paragraph{Boundary commutator drags.} Let $p_r \in P$ and $\partial_s^r \in p_r$. Fix $i, j \in \{ 1, \ldots, n\}$ such that $i < j$. Choose a basepoint $y_s^r \in \partial_s^r$. Let $\gamma_s^r \subset \mathcal{Z}$ be the unique arc (up to isotopy) from $y_s^r$ to $*$. Let $\beta_k^{rs} = \gamma_s^r\alpha_k(\gamma_s^r)^{-1}$ for $k \in \{i, j\}$. Then, we define the \emph{boundary commutator drags} $\BCD_{ij}^{rs} = \Push_{\partial_s^r}([\beta_i^{rs}, \beta_j^{rs}]) \in \Outb{n}{b}$.

It is clear from the definition that $\BCD_{ij}^{rs}$ acts trivially on $\alpha_1, \ldots, \alpha_n$ and arcs that do not have an endpoint on $\partial_s^r$. Suppose that $h$ is an oriented arc with an endpoint on $\partial_s^r$. Without loss of generality, suppose the terminal endpoint of $h$ lies on $\partial_s^r$. Applying lemma \ref{arcform}, we may write $[h] = [\alpha] + [h_0]$, where $\alpha$ is a loop based at $*$ and $h_0 \subset \mathcal{Z}$ is the unique arc (up to isotopy) which shares endpoints with $h$. We just saw that $\BCD_{ij}^{rs}$ fixes the $\alpha_k$, and thus fixes the homology class $[\alpha]$. Therefore,
\begin{align*}
  \BCD_{j\ell m}([h]) &= \BCD_{ij}^{rs}([\alpha] + [h_0])\\
  &= [\alpha] + \BCD_{ij}^{rs}([h_0]) \\ 
  &= [\alpha] + [h_0] + [\alpha_i \cdot \alpha_j \cdot \alpha_i^{-1} \cdot \alpha_j^{-1}] \\
  &= [\alpha] + [h_0] \\
  &= [h]. 
\end{align*}
So, it follows that $\BCD_{ij}^{rs} \in \IOpar{n}{b}{P}$ as well. 

\paragraph{$P$-drags.} The final type of elements we will define are called $P$-drags, where $P$ is a partition of the boundary components of $\XX{n}{b}$. Let $p_r \in P$ and $i \in \{1, \ldots, n\}$. Let $\Sigma_r \subset \mathcal{Z}$ be the unique 2-sphere (up to isotopy) which separates the boundary components of $p_r$ from the remaining boundary components and the $\sigma_j^{\pm}$. Splitting $\XX{n}{b}$ along $\Sigma_r$ gives $\XX{n}{b-c+1} \sqcup \XX{0}{c+1}$, where $c$ is the number of boundary components in $p$. Let $\Sigma_r' \subset \partial\XX{n}{b-c+1}$ be the boundary component coming from this splitting. Just as in the construction of the other drags, fix a basepoint $y_r \in \Sigma_r'$ and an oriented arc $\gamma_r$ from $y_r$ to $*$ to get a basis $\{ \beta_1^r, \ldots, \beta_n^r \}$ of $\pi_1(\XX{n}{b-c+1}, y_r)$. See Figure \ref{pdrag}. Then, we define the \emph{$P$-drag} $\PD_i^r := \iota_*(\Push_{\Sigma_r'}(\beta_i), \id) \in \Outb{n}{b}$, where $\iota_*: \Outb{n}{b-c+1} \times \Outb{0}{c+1} \to \Outb{n}{b}$ is the map induced by splitting along $\Sigma_r$.

To see why $\PD_i^r \in \IOpar{n}{b}{P}$, first notice that we can isotope $\PD_i^r$ to fix all the $\alpha_j$. Next, if $h$ is an arc connecting $P$-adjacent boundary components, we write $[h] = [\alpha] + [h_0]$ as in Lemma \ref{arcform}. As we just noted, $\PD_i^r$ fixes $[\alpha]$, so it suffices to show that $\PD_i^r$ fixes the homology class of $h_0$. If the endpoints of $h$ lie on boundary components in $p_r$, then we may homotope $h_0$ such that it never crosses $\Sigma_r$. Then, $\PD_i^r$ fixes $h_0$. On the other hand, if the endpoints of $h$ lie on boundary components which are not in $p_r$, then again we can homotope $h_0$ such that it does not cross $\Sigma_r$, and then homotope $\PD_i^r$ such that it fixes $h_0$. In either case, $\PD_i^r$ fixes the homology class of of $h_0$, and so we conclude that $\PD_i^r \in \IOpar{n}{b}{P}$. 

\begin{figure}
\centering
\labellist
\small\hair 2pt
\pinlabel $\gamma_p$ at 210 192
\pinlabel $\alpha_2$ at 100 150
\pinlabel $\partial_2^1$ at 282 197
\pinlabel $\partial_1^1$ at 298 222
\pinlabel $\Sigma_p'$ at 235 240
\endlabellist
\includegraphics[width=0.4\textwidth]{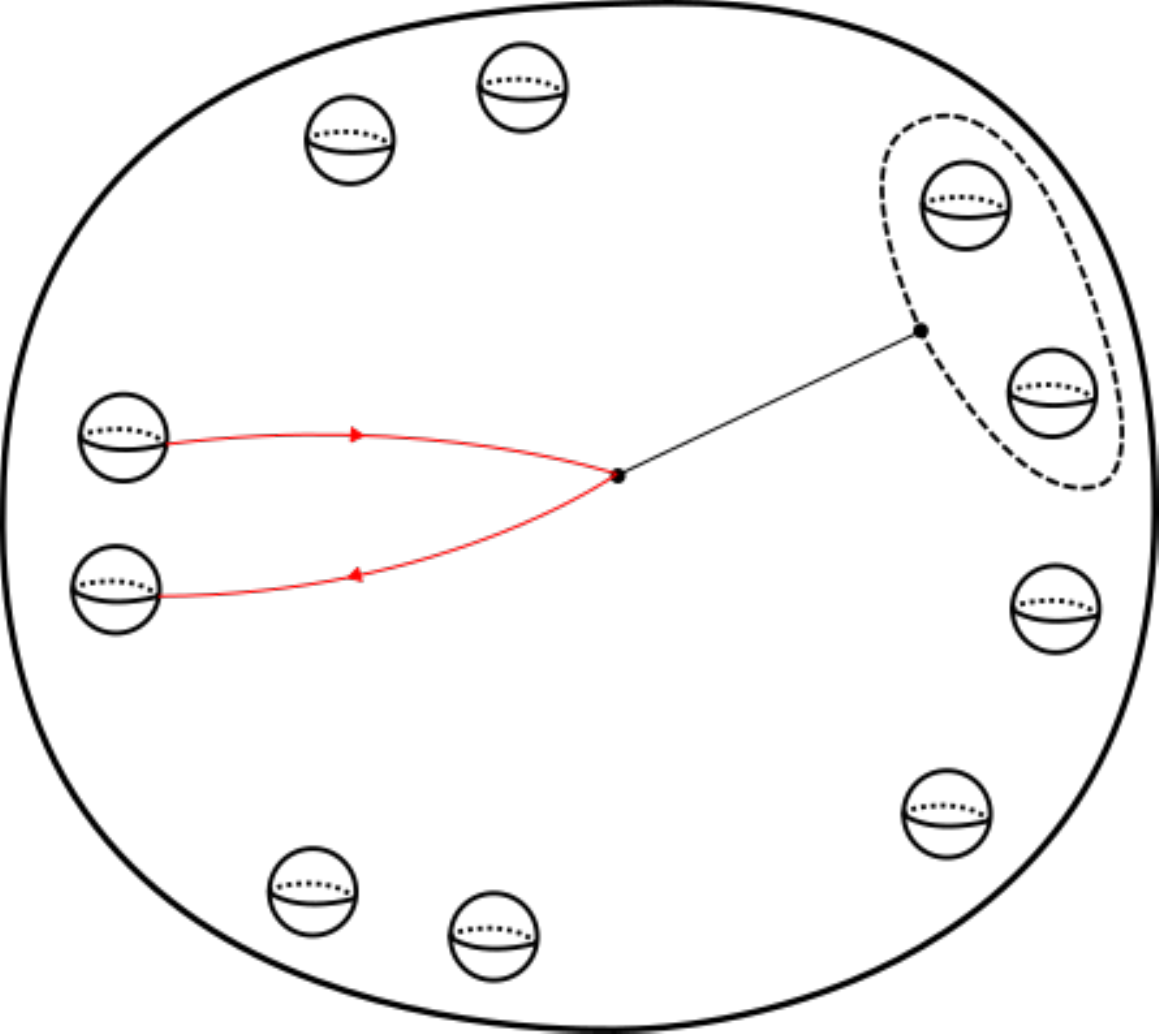}
\caption{Setup of the $P$-drag $\PD_2^p$, where $p = \{\partial_1^1, \partial_2^1\} \in P$.}
\label{pdrag}
\end{figure}
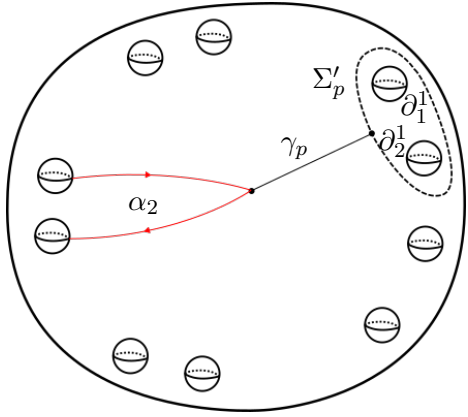

\paragraph{Images under capping.} Suppose we have an embedding $\iota: \XX{n}{b} \hookrightarrow \XX{n}{b-1}$ given by capping off the boundary component $\partial$. Let $\iota_*: \IOpar{n}{b}{P} \to \IOpar{n}{b-1}{P'}$ be the induced map, where $P'$ is the partition of the boundary components of $\XX{n}{b-1}$ induced by $P$. Using the geometric free basis $\{\alpha_1, \ldots, \alpha_n\}$ and corresponding sphere basis $\{S_1, \ldots, S_n\}$ for $\XX{n}{b}$, we get a corresponding geometric free basis $\{ \iota(\alpha_1), \ldots, \iota(\alpha_n) \}$ and sphere basis $\{ \iota(S_1), \ldots, \iota(S_n) \}$ for $\XX{n}{b-1}$. Moreover, the ordering on $P$ (and each $p_r \in P$) induces an ordering on $P'$. We can repeat the process described throughout this section to define handle drags, commutator drags, boundary commutator drag, and $P'$-drags in $\IOpar{n}{b-1}{P'}$, which we will denote by $\overline{\HD}_{ij}$, $\overline{\CD}_{ijk}^{\pm}$, $\overline{\BCD}_{ij}^{rs}$, and $\overline{\PD}_{i}^{r'}$, respectively. With this setup, we find that:
\begin{itemize}
\item $\iota_*(\HD_{ij}) = \overline{\HD}_{ij}$
\item $\iota_*(\CD_{ijk}^{\pm}) = \overline{\CD}_{ijk}^{\pm}$
\item $\iota_*(\BCD_{ij}^{rs}) = \id$ if $\partial = \partial_s^r$ and $\iota_*(\BCD_{ij}^{rs}) = \overline{\BCD}_{ij}^{rs}$ if $\partial \neq \partial_s^r$
\item $\iota_*(\PD_i^r) = \id$ if $p_r = \{\partial_1^r\}$ and $\iota_*(\PD_i^r) = \overline{\PD}_i^{r}$ if $p_r \neq \{\partial_1^r\}$.
\end{itemize}

\section{Finite Generation}\label{finitegenerationsection}
Now that we have defined our collection of candidate generators for $\IOpar{n}{b}{P}$, we now move on to proving that they do in fact generate. The first step in this proof will be an induction on $b$ to reduce to the case of $b=0$. This induction will rely on the following theorem of Tomaszewski \cite{Tomaszewski} (see \cite{PutmanCommutator} for a geometric proof).

\begin{theorem}[Tomaszewski]\label{tomaszewski}
Let $F_n$ be the free group on $n$ letters $\{x_1, \ldots, x_n\}$. The commutator subgroup $[F_n, F_n]$ of $F_n$ is freely generated by the set
\begin{equation*}
\left\{ [x_i, x_j]^{x_i^{d_i} \cdots x_n^{d_n}}, 1 \leq i < j \leq n, d_\ell \in \mathbb{Z}, i \leq \ell \leq n \right\},
\end{equation*}
where the superscript denotes conjugation.
\end{theorem}

We will also need the following lemma from group theory.

\begin{lemma}\label{fingenlemma}
  Consider an exact sequence of groups
  \begin{equation*}
  1 \to K \to G \to Q \to 1.
  \end{equation*}
  Let $S_Q$ be a generating set for $Q$. Moreover, assume that there are sets $S_K \subset K$ and $S_G \subset G$ such that $K$ is contained in the subgroup of $G$ generated by $S_K$ and $S_G$. Then $G$ is generated by the set $S_K \cup S_G \cup \widetilde{S}_Q$, where $\widetilde{S}_Q$ is a set consisting of one lift $\widetilde{q} \in G$ for every element $q \in S_q$.
\end{lemma} 
  
\begin{proof}[Proof of lemma]
  Let $G' \subset G$ be the subgroup generated by $S_K \cup S_G \cup \widetilde{S}_Q$, and let $K' = G' \cap K$. Then the following diagram commutes and has exact rows:
  \begin{equation*}
  \begin{tikzcd}
  1 \arrow[r] & K' \arrow[r] \arrow[d, "\varphi"] & G' \arrow[r] \arrow[d] & Q \arrow[r] \arrow[d, "="] & 1 \\
  1 \arrow[r] & K \arrow[r] & G \arrow[r] & Q \arrow[r] & 1.
  \end{tikzcd}
  \end{equation*}
  The vertical maps are all inclusions, and hence injective. Also, by assumption, the map $\varphi$ is surjective. Therefore, by the five lemma, all of the vertical maps are isomorphisms, and so we are done.
\end{proof}

\noindent We now prove Theorem \ref{finitegeneration} by proving the following stronger result. 

\begin{theorem}\label{strongfinitegeneration}
The group $\IOpar{n}{b}{P}$ is generated by handle, commutator, boundary commutator, and $P$-drags for $b \geq 0$, $n > 0$. 
\end{theorem}

\begin{proof}
As mentioned above, we will prove this by induction on $b$. The base case $b=0$ follows directly from Magnus's Theorem \ref{magnus}. 

If $b > 0$, fix a boundary component $\partial$ of $\XX{n}{b}$ and let $p \in P$ be the partition containing $\partial$. Let $\iota: \XX{n}{b} \hookrightarrow \XX{n}{b-1}$ be an embedding obtained by capping off $\partial$, and choose a basepoint $* \in \XX{n}{b-1} \setminus \XX{n}{b}$. By Theorem \ref{birman}, there is an exact sequence
\begin{equation*}
1 \to L \overset{\Push}{\longrightarrow} \IOpar{n}{b}{P} \overset{\iota_*}{\longrightarrow} \IOpar{n}{b-1}{P'} \to 1,
\end{equation*}
where $L = \pi_1(\XX{n}{b}, *)$ if $p = \{\partial\}$ and $L = [\pi_1(\XX{n}{b}, *), \pi_1(\XX{n}{b}, *)]$ otherwise. As we saw in the discussion at the end of Section \ref{generatorssection}, we can define the drags of $\IOpar{n}{b}{P}$ and $\IOpar{n}{b-1}{P'}$ in a consistent way; that is, we can define our drags in such a way that $\iota_*$ takes handle drags to handle drags, commutator drags to commutator drags, and so on. By induction, $\IOpar{n}{b-1}{P'}$ is generated by the desired drags. Therefore, it suffices to show that $\Push(L)$ is generated by our drags as well. If $p = \{\partial\}$, then $\Push(L)$ is precisely the subgroup of $\IOpar{n}{b}{P}$ generated by the $P$-drags, and so we are done in this case.

The case of $p \neq \{\partial\}$ is less straightforward since the commutator subgroup of a free group is not finitely generated when $n \geq 2$. However, this is not necessary for $\IOpar{n}{b}{P}$ to be finitely generated by our collection of drags. We will appeal to Lemma \ref{fingenlemma}. Suppose that $p \neq \{\partial\}$. Then, by Theorem \ref{tomaszewski}, the kernel $L = [\pi_1(\XX{n}{b}, *), \pi_1(\XX{n}{b}, *)]$ of the Birman exact sequence is generated by elements of the form $[x_i, x_j]^{x_i^{d_i} \cdots x_n^{d_n}}$. First, notice that $\Push([x_i, x_j])$ is the boundary commutator drag $\BCD_{ij}^{rs}$, where $\partial_s^r = \partial$ is the boundary component of $\XX{n}{b}$ being capped off. Moreover, we have seen that the handle drag $\HD_{k\ell}$ acts on $x_k$ by $x_k \mapsto x_\ell x_kx_\ell^{-1}$. It follows that $\HD_{ik} \cdot \HD_{jk} ([x_i, x_j]) = [x_i, x_j]^{x_k}$. Continuing this pattern, we see that 
\begin{equation*}
  [x_i, x_j]^{x_i^{d_i} \cdots x_n^{d_n}} = (\HD_{in} \cdot \HD_{jn})^{d_n} \cdots (\HD_{ii} \cdot \HD_{ji})^{d_i} ([x_i, x_j]),
\end{equation*}
where $\HD_{ii}$ is taken to be trivial. Therefore, 
\begin{equation*}
\Push([x_i, x_k]^{x_i^{d_i} \cdots x_n^{d_n}}) = (\HD_{in} \cdot \HD_{kn})^{d_n} \cdots (\HD_{ii} \cdot \HD_{ki})^{d_i} \cdot \BCD_{ij}^{rs}.
\end{equation*}
This shows that $\Push(L)$ is contained in the subgroup of $\IOpar{n}{b}{P}$ generated by boundary commutator and handle drags. Applying Lemma \ref{fingenlemma} (taking $S_G = \{\text{handle drags}\}$ and $S_K = \{\text{boundary commutator drags}\}$), we conclude that $\IOpar{n}{b}{P}$ is generated by the desired drags. 
\end{proof}

\section{Partial Proof of Magnus's Theorem}\label{magnussection}

In this section, we will give a partial proof of Magnus's Theorem \ref{magnus}, which constituted the base case in the proof of Theorem \ref{strongfinitegeneration}. As stated in the introduction, the original proof of Magnus's Theorem involved two steps: showing that the elements $M_{ij}$ and $M_{ijk}$ \emph{normally} generate $\IO_n$, and then showing that the subgroup generated by these elements is normal. We will give a proof of the first step here (Theorem \ref{magnusweak}).

In order to establish this fact, we will examine the action of $\IOpar{n}{0}{\{\}} = \IO_n$ on a certain simplicial complex, and apply the following theorem of Armstrong \cite{Armstrong}. We say that a group $G$ acts on a simplicial complex $X$ \emph{without rotations} if every simplex $s$ is fixed pointwise by every element of its stabilizer, which we will denote by $G_s$. 

\begin{theorem}[Armstrong]\label{gentool}
Suppose the group $G$ acts on a simply-connected simplicial complex $X$ without rotations. If $X/G$ is simply-connected, then $G$ is generated by the set
\begin{equation*}
\bigcup_{v \in X^{(0)}} G_v.
\end{equation*} 
Here $X^{(0)}$ is the 0-skeleton of $X$.
\end{theorem}

\begin{remark}
  In \cite{Armstrong}, Armstrong proves the converse of this theorem as well. For a modern discussion of the proof of Theorem \ref{gentool}, along with some generalizations, we refer the reader to \cite[Section 3]{PresWithoutChoices}.
\end{remark}
   

\paragraph{The nonseparating sphere complex.} The complex to which we will apply Theorem \ref{gentool} will be the \emph{nonseparating sphere complex} $\nonsepcpx$. Vertices of $\nonsepcpx$ are isotopy classes of smoothly embedded non-nullhomotopic 2-spheres in $\XXnobound{n}$, and $\nonsepcpx$ has a $k$-simplex $\{S_0, \ldots, S_k\}$ if the spheres $S_0, \ldots, S_k$ can be realized pairwise disjointly and their union does not separate $\XXnobound{n}$. This is a subcomplex of the more ubiquitous \emph{sphere complex}, which was introduced by Hatcher in \cite{HomStab} as a tool to explore the homological stability of $\Out(F_n)$ and $\Aut(F_n)$. In \cite[Proposition 3.1]{HomStab}, Hatcher proves the following connectivity result about $\nonsepcpx$.

\begin{proposition}[Hatcher]\label{nonsepconnect}
  The complex $\nonsepcpx$ is $(n-2)$-connected.
\end{proposition}

In particular, $\nonsepcpx$ is simply connected for $n \geq 3$. Recall that sphere twists act trivially on isotopy classes of embedded surfaces, and so we get an action of $\IO_n$ on $\nonsepcpx$. Notice that spheres in a simplex of $\nonsepcpx$ necessarily represent distinct $H_2$-classes. By Poincar\'{e} duality, elements of $\IO_n$ act trivially on $H_2(\XXnobound{n})$, and so this implies that $\IO_n$ acts on $\nonsepcpx$ without rotations. Thus, in order to apply Theorem \ref{gentool}, we must show that $\nonsepcpx/\IO_n$ is simply-connected.

To do this, we will give a description of $\nonsepcpx/\IO_n$ in terms of linear algebra. Fix an identification $H_2(\XXnobound{n}) = \mathbb{Z}^n$. Let $\FS$ be the simplicial complex whose vertices are rank 1 summands of $\mathbb{Z}^n$, and there is a $\ell$-simplex $\{A_0, \ldots, A_\ell\}$ if $A_0 \oplus \cdots \oplus A_\ell$ is a rank $\ell+1$ summand of $\mathbb{Z}^n$. There is a map $\varphi: \nonsepcpx/\IO_n \to \FS$ defined as follows. Let $s \in \nonsepcpx/\IO_n$ be a vertex, and choose a sphere $S \subset \XXnobound{n}$ which represents $s$. As noted above, elements of $\IO_n$ act trivially on $H_2(\XXnobound{n})$. Therefore, the homology class $[S] \in H_2(\XXnobound{n})$ does not depend on the choice of representative $S$. We then define $\varphi(s)$ to be the span of $[S]$ in $H_2(\XXnobound{n})$. It is clear that $\varphi$ extends to simplices.

\begin{lemma}\label{complexiso}
The map $\varphi: \nonsepcpx/\IO_n \to \FS$ is an isomorphism of simplicial complexes.
\end{lemma}

\begin{proof}
Let $\sigma = \{A_0, \ldots, A_\ell\}$ be an $\ell$-simplex of $\FS$. We must show that, up to the action of $\IO_n$, there exists a unique $\ell$-simplex $\tilde{\sigma}$ of $\nonsepcpx$ which projects to $\sigma$. 

Let $v_j \in H_2(\XXnobound{n})$ be a primitive element generating $A_j$ for $0 \leq j \leq \ell$, and extend this to a basis $\{v_0, \ldots, v_{n-1}\}$ for $H_2(\XX{n}{b}) = \mathbb{Z}^n$. In Appendix \ref{homologybasessection}, we will prove Lemma \ref{realizationlemma}, which says that there exists a collection $\{S_0, \ldots, S_{n-1}\}$ of disjoint embedded 2-spheres such that $[S_j] = v_j$ for $0 \leq j \leq n - 1$. Then the simplex $\tilde{\sigma} = \{S_0, \ldots, S_\ell\}$ of $\nonsepcpx$ maps to the $\sigma$ under the composition
\begin{equation*}
  \nonsepcpx \to \nonsepcpx/\IO_n \overset{\varphi}{\to} \FS.
\end{equation*}

We will now show that $\tilde{\sigma}$ is unique up to the action of $\IO_n$. Suppose that $\tilde{\sigma}' = \{S_0', \ldots, S_\ell'\}$ is another simplex of $\nonsepcpx$ which projects to $\sigma$. Since $\tilde{\sigma}$ and $\tilde{\sigma}'$ bother project to $\sigma$, we may order and orient the spheres such that $[S_j] = [S_j']$ for $0 \leq j \leq \ell$. Again by Lemma \ref{realizationlemma}, we can extend $\{ S_1, \ldots, S_\ell\}$ and $\{ S_1', \ldots, S_\ell'\}$ to collections of spheres $\{S_1, \ldots, S_n\}$ and $\{ S_1', \ldots, S_n'\}$ such that $[S_j] = [S_j'] = v_j$ for $0 \leq j \leq n-1$. Notice that splitting $\XXnobound{n}$ along either of these collections reduces $\XXnobound{n}$ to a sphere with $2n$ boundary components. Therefore, there exists some $\mathfrak{f} \in \Mod(\XXnobound{n})$ such that $\mathfrak{f}(S_j) = S_j'$ for all $j$. Let $f \in \Out(F_n)$ be the image of $f$. By construction, $f(\tilde{\sigma}) = \tilde{\sigma}'$. Furthermore, $f$ fixes a basis for homology, and so $f \in \IO_n$. This completes the proof.
\end{proof}

This description of $\nonsepcpx/\IO_n$ is advantageous because $\FS$ is known to be $(n-2)$-connected, and hence simply connected for $n \geq 3$. The first proof of this fact is due to Maazen \cite{Maazen} in his unpublished thesis (see \cite[Theorem E]{ChurchFarbPutman} for a published proof). Thus, we have shown that $\nonsepcpx/\IO_n$ is sufficiently connected.

\begin{corollary}[Maazen]\label{connectivitycor}
  The complex $\nonsepcpx/\IO_n$ is simply connected for $n \geq 3$. 
\end{corollary}

As indicated in Theorem \ref{gentool}, the stabilizers of spheres play an important role in the proof of Theorem \ref{magnusweak}, and so we introduce notation for them here. If $S$ is an isotopy class of embedded sphere in $\XXnobound{n}$, we denote by $\Out(F_n, S)$ the stabilizer of $S$ in $\Out(F_n)$, and define $\IO_n(S) = \Out(F_n, S) \cap \IO_n$. We now move on to the proof of Theorem \ref{magnusweak}.
 

\begin{proof}[Proof of Theorem \ref{magnusweak}]
  We will induct on $n$. The base cases are easy; $\IO_1$ and $\IO_2$ are both trivial. Suppose now that $\IO_{n-1}$ is $\Out(F_{n-1})$-normally generated by handle and commutator drags. We must now show that $\IO_n$ is $\Out(F_n)$-normally generated by handle and commutator drags as well. By Theorem \ref{gentool}, Proposition \ref{nonsepconnect}, and Corollary \ref{connectivitycor}, it suffices to show that $\IO_n(S)$ is generated by $\Out(F_n)$-conjugates of these drags for all $S$. Let $\{\alpha_1, \ldots, \alpha_n\}$ be the geometric free basis of $\pi_1(\XXnobound{n})$ identified with our fixed generating set $\{x_1, \ldots, x_n\}$ of $F_n$, and let $\{S_1, \ldots, S_n\}$ be a corresponding sphere basis. Use these bases to construct the handle and commutator drags as in Section \ref{generatorssection}. Recall that handle drags correspond to the automorphisms $M_{ij}$ of Magnus's generators, and commutator drags correspond to $M_{ijk}$. We will first show that $\IO_n(S_1)$ is $\Out(F_n, S_1)$-normally generated by handle and commutator drags.

  Splitting $\XXnobound{n}$ along $S_1$ yields a copy of $\XX{n-1}{2}$. Let $N$ be the tubular neighborhood of $S_1$ removed in this splitting, and let $\partial_1$ and $\partial_2$ be the boundary components of $\XX{n-1}{2}$. Then this splitting induces a surjective map $\Outb{n-1}{2} \to \Out(F_n, S_1)$, which restricts to a  map $\iota_*: \IOpar{n-1}{2}{P} \to \IO_n(S_1)$, where $P = \{p_1\} = \{\{\partial_1, \partial_2\}\}$. This map is also surjective.

  Use the bases $\{\alpha_2, \ldots, \alpha_n\}$ and $\{S_2, \ldots, S_n\}$ to construct the handle, commutator, boundary commutator, and $P$-drags in $\IOpar{n-1}{2}{P}$. By our induction hypothesis combined with the proof of Theorem \ref{strongfinitegeneration}, these drags $\Outb{n-1}{2}$-normally generate $\IOpar{n-1}{2}{P}$. Notice that with these choices of drags, the map $\iota_*$ takes handle and commutator drags to handle and commutator drags. Moreover, $\iota_*$ takes boundary commutator drags in $\IOpar{n-1}{2}{P}$ to commutator drags in $\IO_n(S)$, and takes the $P$-drag $\PD_i^{P}$ to the handle drag $\HD_{1i}$. Thus, $\IO_n(S_1)$ is $\Out(F_n, S_1)$-normally generated by handle and commutator drags.

  Finally, let $S$ be an arbitrary vertex of $\nonsepcpx$. Since $S$ is nonseparating, there exists some $f \in \Out(F_n)$ such that $f(S_1) = S$. It follows that 
  \begin{equation*}
    \IO_n(S) = f \cdot \IO_n(S_1) \cdot f^{-1}.
  \end{equation*}
  Since $\IO_n(S_1)$ is $\Out(F_n, S_1)$-normally generated by handle and commutator drags, it follows that $\IO_n(S)$ is generated by $\Out(F_n)$-conjugates of handle and commutator drags, which is what we wanted to show. 
\end{proof}

\section{Computing the abelianization}\label{abelianizationsection}
In this section, we compute the abelianization of the group $\IOpar{n}{b}{P}$, proving Theorem \ref{abelianizationtheorem}. For the Torelli subgroup of the mapping class group of a surface, this was done by Johnson \cite{Johnson}. Some key tools used in this computation are the \emph{Johnson homomorphisms}
\begin{equation*}
  \tau_{\Sigma_{g,1}}: \Tor(\Sigma_{g,1}) \to \wedge^3 H \qquad \text{and} \qquad \tau_{\Sigma_{g}}: \Tor(\Sigma_{g}) \to (\wedge^3 H) / H,
\end{equation*}
where $H = H_1(\Sigma_{g,b})$. Johnson showed that these homomorphisms are the abelianization maps modulo torsion if $g \geq 3$. For $\IA_n = \IOpar{n}{1}{}$, Andreadakis \cite{Andreadakis} and Bachmuth \cite{Bachmuth} used an analogous homomorphism $\tau: \IA_n \to \Hom(H, \wedge^2 H)$ (where now $H = H_1(F_n) = \mathbb{Z}^n$) to show that
\begin{equation*}
  H_1(\IA_n) \cong \Hom(H, \wedge^2 H) \cong \mathbb{Z}^{n \cdot \binom{n}{2}}.
\end{equation*}
We will begin by recalling the definition of $\tau$, along with the computation of the ranks of $H_1(\IA_n)$ and $H_1(\IO_n)$, and then proceed to the case of multiple boundary components.
 
\paragraph{The Johnson homomorphism.} Recall that $\Outb{n}{1} \cong \Aut(F_n)$, and the subgroup $\IA_n$ is precisely those automorphisms which act trivially on $H_1(F_n) = \mathbb{Z}^n$. The goal is to construct a homomorphism $\tau: \IA_n \to \Hom(H, \wedge^2 H)$, where $H = H_1(F_n) = \mathbb{Z}^n$. 

First, we claim that the group $[F_n, F_n]/[F_n, [F_n, F_n]]$ is isomorphic to $\wedge^2 H$, where $[F_n, F_n]$ denotes the commutator subgroup of $F_n$. To see this, consider the short exact sequence
\begin{equation*}
  1 \to [F_n, F_n] \to F_n \to \mathbb{Z}^n \to 1.
\end{equation*}
Passing to the five-term exact sequence in homology, we get the sequence
\begin{equation*}
  0 \to H_2(\mathbb{Z}^n) \to H_1([F_n, F_n])_{\mathbb{Z}^n} \to H_1(F_n) \to H_1(\mathbb{Z}^n) \to 0,
\end{equation*}
where $H_1([F_n, F_n])_{\mathbb{Z}^n} = [F_n, F_n]/[F_n, [F_n, F_n]]$ denotes the group of co-invariants of $H_1([F_n, F_n])$ with respect to the action of $\mathbb{Z}^n$ (induced by the conjugation action of $F_n$ on $[F_n, F_n]$). The map $H_1(F_n) \to H_1(\mathbb{Z}^n)$ is clearly an isomorphism, and so it follows that the map $H_2(\mathbb{Z}^n) \to [F_n, F_n]/[F_n, [F_n, F_n]]$ is an isomorphism as well. This proves our claim because $H_2(\mathbb{Z}^n) \cong \wedge^2 \mathbb{Z}^n$. Let $\rho: [F_n, F_n] \to \wedge^2 \mathbb{Z}^n$ be the projection. Following the definitions above, we see that $\rho$ is defined by 
\begin{equation*}
  \rho([x,y]) = [x] \wedge [y],
\end{equation*}
where $[x]$ and $[y]$ denote the classes of $x$ and $y$ in $H$, respectively. 
  
Next, let $f \in \IA_n$. Then $f(x)x^{-1}$ is nullhomologous for all $x \in F_n$, and therefore lies in $[F_n, F_n]$. We define the map $\hat{\tau}_f: F_n \to \wedge^2 H$ via
\begin{equation*}
  \hat{\tau}_f(x) = \rho(f(x)x^{-1}).
\end{equation*}
We now check that $\hat{\tau}_f$ is a homomorphism. Let $x, y \in F_n$. Applying the relation $ab = [a,b]ba$, we get
\begin{align*}
  \hat{\tau}_f(xy) &= \rho(f(x)f(y)y^{-1}x^{-1}) \\
  &= \rho \left( [f(x), f(y)y^{-1}] \cdot (f(y)y^{-1}) \cdot (f(x)x^{-1}) \right) \\
  &= [f(x)] \wedge [f(y)y^-1] + \hat{\tau}_f(y) + \hat{\tau}_f(y) \\
  &= \hat{\tau}_f(y) + \hat{\tau}_f(y),
\end{align*}
since $[f(y)y^{-1}] = 0$. This shows that $\hat{\tau}$ is indeed a homomorphism. Furthermore, since $\wedge^2 H$ is abelian, the map $\hat{\tau}: F_n \to \wedge^2 H$ factors through the abelianization, inducing a map $\tau_f: H \to \wedge^2 H$. Therefore, we have a map 
\begin{equation*}
  \tau: \IA_n \to \Hom(H, \wedge^2 H)
\end{equation*}
sending $f$ to $\tau_f$. We now check that $\tau$ is a homomorphism. Let $f, g \in \IA_n$. Then 
\begin{align*}
  \tau_{fg}([x]) &= \rho(f(g(x))x^{-1}) \\
  &= \rho(f(g(x))(g(x))^{-1}g(x)x^{-1}) \\
  &= \tau_f([g(x)]) + \tau_g([x]) \\
  &= \tau_f([x]) + \tau_g([x])
\end{align*}
since $g$ fixes $[x]$. Thus, $\tau$ is the desired homomorphism.

We now move on to computing the image of our generators under $\tau$. Since we are dealing with the case of one boundary component, boundary commutator drags are unnecessary since they are products of $P$-drags. Fix a basepoint $* \in \partial \XX{n}{1}$, and choose a basis $\{x_1, \ldots, x_n\}$ of $\pi_1(\XX{n}{1}, *)$. Construct the handle, commutator, and $P$-drags with respect to this basis. 
 
\paragraph{Handle drags.} Recall that the handle drag $\HD_{ij}$ acts on $\pi_1(\XX{n}{1})$ by sending $x_i$ to $x_jx_ix_j^{-1}$, and fixing the remaining basis elements. Therefore, 

\begin{align*}
  \tau(\HD_{ij})([x_\ell]) = \rho(\HD_{ij}(x_\ell)x_\ell^{-1}) = \begin{cases} 
    0 & \text{if } \ell \neq i \\
    \rho(x_jx_ix_j^{-1}x_i^{-1}) & \text{if } \ell = i.
  \end{cases}
\end{align*}
Thus, $\tau(\HD_{ij})$ is the homomorphism $[x_i] \mapsto [x_j] \wedge [x_i]$ (and all other generators are sent to $0$). 

\paragraph{Commutator drags.} Notice that the product of commutator drags $\CD_{ijk}^+ \cdot \CD_{ijk}^-$ is equal to a commutator of handle drags. Therefore, only the $\CD_{ijk}^-$ are needed in our generating set, and we can disregard the $\CD_{ijk}^+$ from now on. Recall that $\CD_{ijk}^-$ acts on $\pi_1(\XX{n}{1})$ by sending $x_i$ to $[x_j, x_k]x_i$. Therefore, 
\begin{align*}
  \tau(\CD_{ijk}^-)([x_\ell]) = \rho(\CD_{ijk}^-(x_\ell)x_\ell^{-1}) = \begin{cases}
    0 & \text{if } \ell \neq i \\
    \rho([x_j, x_k]) & \text{if } \ell = i.
  \end{cases}
\end{align*}
It follows that $\tau(\CD_{ijk}^-)$ is the map $[x_i] \mapsto [x_j] \wedge [x_k]$.

\paragraph{$P$-drags.} Next, we note that the product
\begin{equation}\label{eq:trivialproduct}
  \PD_j \cdot \HD_{1j} \cdots \HD_{nj}
\end{equation}
is trivial in $\IA_n$. For a justification of this fact, see the proof of the claim at the end of Theorem \ref{injtheorem}. It follows that the $P$-drags are also redundant in our generating set for $\IA_n$, and can be removed.

\paragraph{Abelianization of $\IA_n$.} To compute the rank of the abelianization of $\IA_n$, we use the following lemma.

\begin{lemma}\label{lem:rankofabelianization}
  Let $G$ be a group and $S$ a finite generating set for $G$. Suppose that $\varphi: G \to \mathbb{Z}^{\vert S \vert}$ is a surjective homomorphism. Then $H_1(G) \cong \mathbb{Z}^{\vert S \vert}$. 
\end{lemma}

\begin{proof}
  Let $F(S)$ denote the free group on the set $S$. Since $\mathbb{Z}^{\vert S \vert}$ is abelian, the homomorphism $\varphi$ factors through the abelianization to give a map $\overline{\varphi}: H_1(G) \to \mathbb{Z}^n$, which is also surjective. Additionally, by the universal property of free groups, we have a map $\psi: F(S) \to G$. Passing to the abelianizations induces a map $\overline{\psi}: H_1(F(S)) \to H_1(G)$. Since $S$ is a generating set for $G$, this map is also surjective. It follows that $\overline{\varphi} \circ \overline{\psi}$ is a surjective map between free abelian groups of equal rank, and is hence an isomorphism. Thus, $\overline{\varphi}$ is an isomorphism as well. 
\end{proof}

From the discussion in the preceeding paragraphs, we have a generating set for $\IA_n$ of size 
\begin{equation*}
  \#(\text{Handle Drags}) + \#(\text{Commutator Drags}) = n(n-1) + n \cdot \binom{n-1}{2} = n \cdot \binom{n}{2}
\end{equation*}
(since $P$-drags can be written as a product of handle drags), and the image of this generating set spans $\Hom(H, \wedge^2 H)$, which also has dimension $n \cdot \binom{n}{2}$. Therefore, by Lemma \ref{lem:rankofabelianization}, the group $H_1(\IA_n)$ has rank $n \cdot \binom{n}{2}$. 

To compute the rank of $H_1(\IO_n)$, consider the quotient map $\IA_n \to \IO_n$, whose kernel is the subgroup of inner automorphisms (or $P$-drags under our geometric interpretation of $\IA_n$). We compute the image of a $P$-drag under $\tau$:
\begin{equation*}
  \tau(\PD_i)([x_\ell]) = \rho(\PD_i(x_\ell)x_\ell^{-1}) = \rho(x_i^{-1}x_\ell x_ix_\ell^{-1}) = [x_\ell] \wedge [x_i].
\end{equation*}
Since $\tau(\PD_i)$ is nontrivial, $\tau$ does not descend to a map $\IO_n \to \Hom(H, \wedge^2 H)$. However, the images $\{\tau(\PD_i)\}$ span a subgroup of $\Hom(H, \wedge^2 H)$ isomorphic to $H$ (where the isomorphism is given by $[h] \mapsto ([x_\ell] \mapsto [x_\ell] \wedge [h])$). So, $\tau$ induces a map $\IO_n \to \Hom(H, \wedge^2 H)/H$. Just as the element given in Equation (\ref{eq:trivialproduct}) is trivial in $\IA_n$, the product $\HD_{1j} \cdots \HD_{nj}$ is trivial in $\IO_n$ for all $j \in \{1, \ldots, n\}$. Thus, we may throw out $n$ handle drags from our generating set to obtain a generating set for $\IO_n$ of size $n \cdot \binom{n}{2} - n$. Since $\Hom(H, \wedge^2 H)$ has rank $n \cdot \binom{n}{2} - n$, Lemma \ref{lem:rankofabelianization} implies that $H_1(\IO_n)$ has rank $n \cdot \binom{n}{2} - n$ as well. This verifies Theorem \ref{rankabel} from the introduction.

\paragraph{Multiple boundary components.} We now move on to the case of multiple boundary components. Just as we did when constructing our drags in Section \ref{generatorssection}, fix an ordering $P = \{p_1, \ldots, p_{\vert P \vert}\}$ and an ordering $p_r = \{\partial_1^r, \ldots, \partial_{b_r}^r\}$ for each $p_r \in P$. We cap off the boundary components of each $p \in P$ as follows:
\begin{itemize}
  \item If $\vert p \vert = 1$, we attach a copy of $\XX{1}{1}$ to the single boundary component of $p$.
  \item If $\vert p \vert = k > 1$, we attach a copy of $\XX{0}{k}$ to these $k$ boundary components.
  \item If $p = p_1$, we follow the same rules as above, except we introduce an additional boundary component in the piece glued to $p$. 
\end{itemize}
Capping off the boundary components in this way gives an embedding
\begin{equation*}
  \iota: \XX{n}{b} \hookrightarrow \XX{m}{1},
\end{equation*}
where the boundary component of $\XX{m}{1}$ lies in the piece attached to $p_1$. This embedding induces a map $\iota_*: \IOpar{n}{b}{P} \to IA_m$. We obtain a similar map $\IA_m \to \IO_m$ by attaching a disk to the boundary component of $\XX{m}{1}$. In Appendix \ref{injectivity}, we will prove Theorem \ref{injtheorem}, which says that the composition
\begin{equation*}
  \IOpar{n}{b}{P} \overset{\iota_*}{\to} \IA_m \to \IO_m
\end{equation*}
is injective. It follows that $\iota_*$ is injective as well. Therefore, to compute the rank of the abelianization of $\IOpar{n}{b}{P}$, it suffices to compute the rank of the abelianization of its image in $\IA_m$. Let $H = H_1(\XX{m}{1})$, and let $\tau_*: \IOpar{n}{b}{P} \to \Hom(H, \wedge^2 H)$ denote the composition
\begin{equation*}
  \IOpar{n}{b}{P} \overset{\iota_*}{\to} \IA_m \overset{\tau}{\to} \Hom(H, \wedge^2 H). 
\end{equation*}
Our goal now becomes computing the images of handle, commutator, boundary commutator, and $P$-drags under $\tau_*$.

\paragraph{Choosing a basis.} To carry out this computation, it will be helpful to choose bases for $\pi_1(\XX{n}{b})$ and $\pi_1(\XX{m}{1})$ carefully. For simplicity, we will assume that $\vert p_1 \vert > 1$. The case of $\vert p_1 \vert = 1$ is more straightforward. Fix a basepoint $z_1^1 \in \partial_1^1$ and a basis $\{x_1, \ldots, x_n\}$ for $\pi_1(\XX{n}{b}, z_1^1)$. Choose a corresponding sphere basis $\{S_1, \ldots, S_n\}$. We define our drags in $\IOpar{n}{b}{P}$ with respect to these bases. 

Next, choose a basepoint $z \in \partial\XX{m}{1}$, and an oriented arc $\alpha_1 \subset \XX{m}{1} \setminus \Int(\XX{n}{b})$ from $z$ to $z_1^1$ (this is possible since the boundary component of $\XX{m}{1}$ lies on the piece attached to $p_1$). For $i \in \{1, \ldots, n\}$, let $y_i = \alpha_1 x_i \alpha_1^{-1}$. Then $\{y_1, \ldots, y_n\}$ is a partial basis for $\pi_1(\XX{m}{1}, z)$. We wish to extend this to a full basis. Throughout the definition of this extended basis, we encourage the reader to follow along in Figure \ref{extendedbasis}.

For each boundary component $\partial_s^r$ of $\XX{n}{b}$, fix a point $z_s^r \in \partial_s^r$ (leaving $z_1^1$ as before). For each $z_s^r \neq z_1^1$, let $\beta_s^r$ be the unique oriented arc (up to isotopy) in $\XX{n}{b} \setminus \bigcup S_i$ from $z_1^1$ to $z_s^r$. For $s \in \{2, \ldots, b_1\}$, define $y_s^1 = \alpha_1\beta_s^1\alpha_s^{-1}$. In Figure \ref{extendedbasis}, the loops $y_1^1$ and $y_2^1$ are of this form.

Next, let $r > 1$. If $\vert p_r \vert = 1$, let $\gamma_1^r$ be an oriented loop based at $z_1^r$ which generates the fundamental group of the copy of $\XX{1}{1}$ attached to $p_r$. Then we define $y_1^r = \alpha_1\beta_1^r\gamma_1^1(\beta_1^r)^{-1}(\alpha_1)^{-1}$. In Figure \ref{extendedbasis}, the curve $y_1^3$ is an example of such a loop. 

On the other hand, suppose $\vert p_r \vert > 1$. For $s \in \{2, \ldots, b_r\}$, let $\gamma_s^r$ be the unique (up to isotopy) oriented curve in $\XX{m}{1} \setminus \interior(\XX{n}{b})$ from $z_1^r$ to $z_s^r$. Then, define 
\begin{equation*}
  y_s^r = \alpha_1\beta_1^r\gamma_s^r(\beta_s^r)^{-1}(\alpha_1)^{-1}.
\end{equation*}
The curve $y_2^2$ is an example of this type of loop in Figure \ref{extendedbasis}. 

\begin{figure}
  \centering
  \labellist
    \small\hair 2pt
    \pinlabel $\alpha_1$ at 95 280

    \pinlabel $y_1$ at 280 105
    \pinlabel $y_2$ at 370 100
    \pinlabel $y_3$ at 430 143

    \pinlabel $y_1^1$ at 85 200
    \pinlabel $y_2^1$ at 65 160

    \pinlabel $y_2^2$ at 320 393

    \pinlabel $y_1^3$ at 550 360

    \large
    \pinlabel $\XX{3}{6}$ at 530 70
  \endlabellist
  \includegraphics[width=0.80\textwidth]{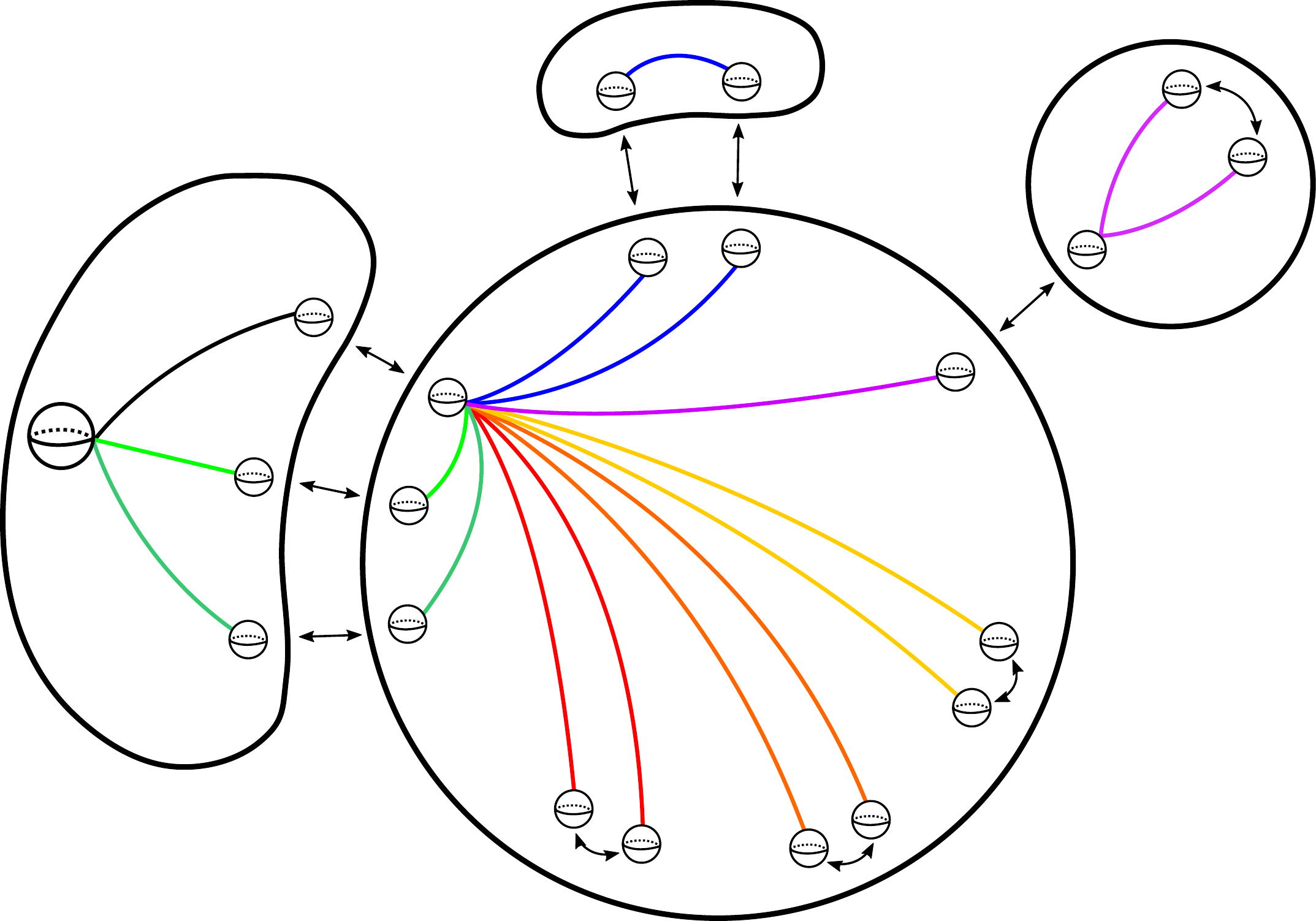}
  \caption{$\XX{3}{6}$ with the partition $P = \{\{\partial_1^1, \partial_2^1, \partial_3^1\}, \{\partial_1^2, \partial_2^2\}, \{\partial_1^3\}\}$ embedded into $\XX{7}{1}$. The loops $\{y_1, y_2, y_3\}$ are freely homotopic to a basis for $\pi_1(\XX{3}{6})$, and this basis has been extended to a basis $\{y_1, y_2, y_3, y_1^1, y_2^1, y_2^2, y_1^3\}$ of $\pi_1(\XX{7}{1}, z)$.}
  \label{extendedbasis}
\end{figure}
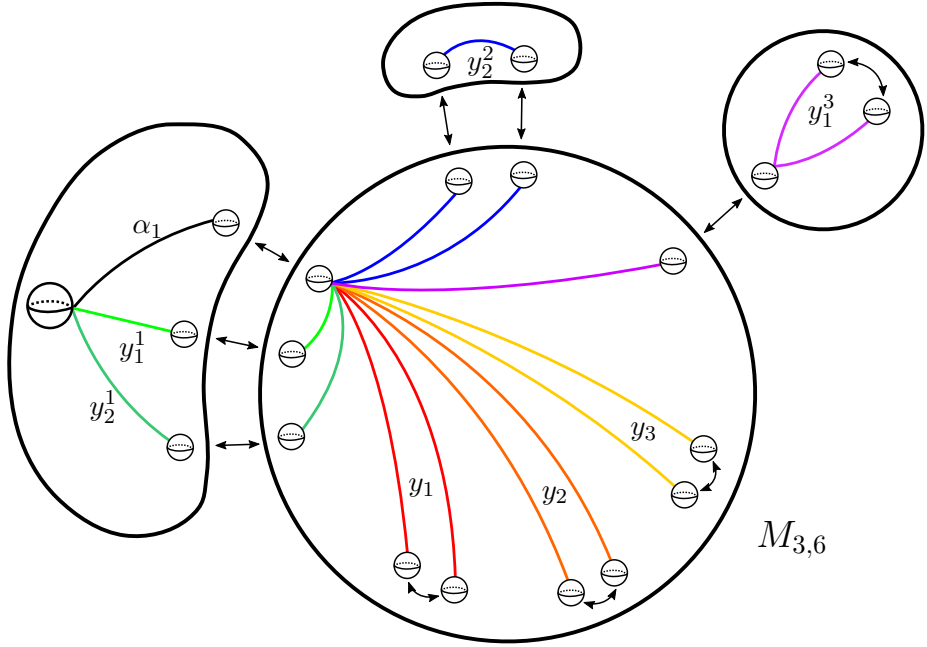

Let $Y = \{y_1 \ldots, y_n\}$. For $r \in \{1, \ldots, \vert P \vert\}$, let $Y_r = \{y_1^r\}$ if $\vert p_r \vert = 1$, and $Y_r = \{y_2^r, \ldots, y_{b_r}^r\}$ otherwise. Then the collection
\begin{equation*}
  Y \cup Y_1 \cup \cdots \cup Y_{\vert P \vert}
\end{equation*}
forms a free basis for $\pi_1(\XX{m}{1}, z)$. 

\paragraph{Computations and relations.} We now move on to the computation of the images of our collection of drags under $\tau_*$. These computations are straightforward, and are summarized in Figure \ref{comptable}. We see from these computations that there is a relation between the images of $P$-drags and Handle Drags. Namely,
\begin{equation}\label{PDrelations}
  \sum_{r=1}^{\vert P \vert} \tau_*(\PD_j^r) = -\sum_{i=1}^n \tau_*(\HD_{ij})
\end{equation}
for all $j \in \{1, \ldots, n\}$. As we saw in the case of one boundary component, this is because
\begin{equation}\label{PDrelationsIO}
  \PD_j^1 \cdots \PD_j^{\vert P \vert} \cdot \HD_{1j} \cdot \HD_{nj} = 1
\end{equation}
in $\IOpar{n}{b}{P}$. Additionally, we see a relation between the image of boundary commutator drags:
\begin{equation}\label{BCDrelations}
  \sum_{s=1}^{b_r} \tau_*(\BCD_{ij}^{rs}) = 0
\end{equation}
for all $r \in \{1, \ldots, \vert P \vert \}$ and $i,j \in \{1, \ldots, n\}$ with $i < j$. This relation holds because 
\begin{equation}\label{BCDrelationsIO}
  \BCD_{ij}^{r1} \cdots \BCD_{ij}^{rb_r} = [\PD_i^r, \PD_j^r]
\end{equation}
in $\IOpar{n}{b}{P}$.

\begin{figure}
  \centering
  \begin{tabular}{|lc|lc|lc|} \hline
    Drag &&  Action on $\pi_1$ && Image under $\tau_*$ &\\ \hline \hline
    $\HD_{ij}$ && $y_i \mapsto y_jy_iy_j^{-1}$ && $[y_i] \mapsto [y_j] \wedge [y_i]$ & \\ \hline
    $\CD_{ijk}^-$ && $y_i \mapsto [y_j, y_k]y_i$ && $[y_i] \mapsto [y_j] \wedge [y_k]$ & \\ \hline
    $\BCD_{jk}^{rs}$ & ($r,s > 1$) & $y_s^r \mapsto y_s^r[y_j, y_k]^{-1}$ && $[y_s^r] \mapsto [y_k] \wedge [y_j]$ & \\ \hline
    $\BCD_{jk}^{r1}$ & ($r > 1$) & $y_s^r \mapsto [y_j, y_k]y_s^r$ & ($s > 1$) & $[y_s^r] \mapsto [y_j] \wedge [y_k]$ & ($s > 1$) \\ \hline
    $\BCD_{jk}^{1s}$ & ($s > 1$) & $y_s^1 \mapsto y_s^1[y_j, y_k]$ & & $[y_s^1] \mapsto [y_j] \wedge [y_k]$ & \\ \hline
    $\BCD_{jk}^{11}$ & & $y \mapsto [y_j, y_k]^{-1}y[y_j, y_k]$ & ($y \not\in Y_1$) & $[y] \mapsto 0$ & ($y \not\in Y_1$) \\
    & &  $y_s^1 \mapsto [y_j, y_k]^{-1}y_s^r$ & ($s > 1$) & $[y_s^1] \mapsto [y_k] \wedge [y_j]$ & ($s > 1$) \\ \hline
    $\PD_j^r$ & ($r > 1$) & $y_s^r \mapsto y_jy_s^ry_j^{-1}$ & ($s > 1$) & $[y_s^r] \mapsto [y_j] \wedge [y_s^r]$ & ($s > 1$) \\ \hline
    $\PD_j^1$ & & $y \mapsto y_j^{-1}y y_j$ & ($y \not\in Y_1$) & $[y] \mapsto [y] \wedge [y_j]$ & ($y \not\in Y_1$) \\ \hline
  \end{tabular}
  \caption{Computing the image of drags under $\tau_*$.}
  \label{comptable}
\end{figure}
 
\paragraph{Contributions to abelianization.} From the computations and relations above, we see that the handle drags and commutator drags together still contribute $n \cdot \binom{n}{2}$ dimensions to the abelianization of $\IOpar{n}{b}{P}$. There are $b \cdot \binom{n}{2}$ boundary commutator drags, but the relations given in Equation (\ref{BCDrelations}) kill off $\vert P \vert \cdot \binom{n}{2}$ of these in the abelianization (though we can also remove this many elements from our generating set by using Equation (\ref{BCDrelationsIO})). Finally, the number of $P$-drags is $\vert P \vert \cdot n$, but $n$ of them are killed in the abelianization by Equation (\ref{PDrelations}) (and again, we may remove $n$ elements from our generating set by Equation (\ref{PDrelationsIO})). Adding these all together, we find that the image of $\tau_*: \IOpar{n}{b}{P} \to \Hom(H, \wedge^2 H)$ has rank
\begin{equation*}
  R = n \cdot \binom{n}{2} + \left(b \cdot \binom{n}{2} - \vert P \vert \cdot \binom{n}{2}\right) + (\vert P \vert \cdot n - n).
\end{equation*}
Moreover, we can reduce our generating set (using Equations (\ref{PDrelationsIO}) and (\ref{BCDrelationsIO})) to a set of size $R$  as well. Thus, by Lemma \ref{lem:rankofabelianization}, the group $H_1(\IOpar{n}{b}{P})$ has rank $R$, which proves Theorem \ref{abelianizationtheorem}.

\appendix
\section{Injectivity of the inclusion map}\label{injectivity}

We end this paper with a proof of the following facts, which are surely known to experts, but for which we do not know a reference. They are significant because they allow us to realize the groups $\Outb{n}{b}$ (and hence $\IOpar{n}{b}{P}$) as subgroups of $\Out(F_m)$. We will begin with a low-genus case.

\begin{lemma}\label{injlemma}
  The induced map $\iota_*: \Outb{1}{1} \to \Out(F_m)$ is injective for any embedding $\iota: \XX{1}{1} \hookrightarrow \XXnobound{m}$.
\end{lemma}

\begin{proof}
By Laudenbach \cite{Laudenbach2}, the group $\Outb{1}{1} \cong \Aut(F_1) \cong \mathbb{Z}/2$, where the nontrivial element $f \in \Outb{1}{1}$ acts on $\pi_1(\XX{1}{1}, x) \cong \mathbb{Z}$ by inverting the generator. Therefore, $\iota_*(f) \in \Out(F_m)$ is the class of the automorphism
\begin{equation*}
  \begin{cases}
    x_1 \mapsto x_1^{-1} \\
    x_j \mapsto x_j & \text{if } j > 1.
  \end{cases}
\end{equation*}
This automorphism is not an inner automorphism for any $m \geq 1$, so $i_*$ is injective. 
\end{proof}

\begin{theorem}\label{injtheorem}
Fix $n, b \geq 1$ such that $(n, b) \neq (1, 1)$, and let $\iota: \XX{n}{b} \hookrightarrow \XXnobound{m}$ be an embedding. The induced map $\iota_*: \Outb{n}{b} \to \Out(F_m)$ is injective if and only if no component of $\XXnobound{m} \setminus \interior (\XX{n}{b})$ is diffeomorphic to a 3-disk. 
\end{theorem}

\begin{proof}
Suppose first that some component of $\XXnobound{m} \setminus \interior (\XX{n}{b})$ is diffeomorphic to a disk, and let $\partial$ be the boundary component of $\XX{n}{b}$ capped off by this disk. By the Birman exact sequence (Theorem \ref{birmanout}), dragging this boundary component along any nontrivial loop will give a nontrivial element in the kernel of $\iota_*$. 

Suppose now that no component of $\XXnobound{m} \setminus \interior (\XX{n}{b})$ is a disk. We will first prove the theorem in the case $b=1$, and then move on to the general result. 

\underline{\textbf{Case 1:}} Suppose we have an embedding $\iota: \XX{n}{1} \hookrightarrow \XXnobound{m}$. Since no component of $\XXnobound{m} \setminus \interior(\XX{n}{b})$ is a disk, $m > n$. If $n=1$, then we are done by Lemma \ref{injlemma}, so we may assume that $n > 1$. Fix a basepoint $x$ on the boundary of $\XX{n}{1}$, and choose a free basis $\{x_1, \ldots, x_n\}$ of $\pi_1(\XX{n}{1}, x)$. The embedding $\iota$ induces an injection $\pi_1(\XX{n}{b}, x) \hookrightarrow \pi_1(\XXnobound{m}, x)$ which identifies $\pi_1(\XX{n}{1}, x)$ with a free summand of $\pi_1(\XXnobound{m}, x)$. This allows us to extend $\{ x_1, \ldots, x_n \}$ to a free basis $\{ x_1, \ldots, x_m \}$ of $\pi_1(\XXnobound{m}, x)$. Given $f \in \Outb{n}{1} \cong \Aut(F_n)$, the image $\iota_*(f) \in \Out(F_m)$ is the class of the automorphism $\varphi \in \Aut(F_m)$ generated by
\begin{equation*}
  \varphi:
  \begin{cases}
    x_i \mapsto f(x_i) & \text{ if } 1 \leq i \leq n \\
    x_i \mapsto x_i & \text{ if } n < i \leq m.
  \end{cases}
\end{equation*}
Suppose that $\varphi$ is an inner automorphism. If $m > n+1$, then $\varphi$ fixes at least two generators of $F_m$, and thus must be trivial. It follows that $f$ is trivial as well. On the other hand, if $m = n + 1$, then $\varphi$ fixes $x_m$. Since $\varphi$ is inner, $\varphi$ must conjugate by a power of $x_m$. However, if $\varphi$ conjugates by a nontrivial power of $x_m$, then $f$ would not act as an automorphism on $\langle x_1, \ldots, x_n \rangle \subset F_m$, which is a contradiction. Thus, $\varphi$ is trivial, and so $f$ is trivial as well.

In summary, we have shown that $\varphi$ is an inner automorphism if and only if $f$ is trivial, which implies that $\iota_*$ is injective. 

\begin{figure}[t]
  \centering
  \labellist
  \small\hair 2pt
  \pinlabel $x$ at 214 240

  \pinlabel $\partial_1$ at 330 155
  \pinlabel $\partial_2$ at 345 205
  \pinlabel $\partial_3$ at 330 300

  \pinlabel $x_3$ at 490 112
  \pinlabel $x_4$ at 467 320
  \pinlabel $x_5$ at 480 408


  \large 
  \pinlabel $\Sigma$ at 245 100
  \pinlabel $\XX{2}{3}$ at 140 230
  \endlabellist
  \includegraphics[width=0.75\textwidth]{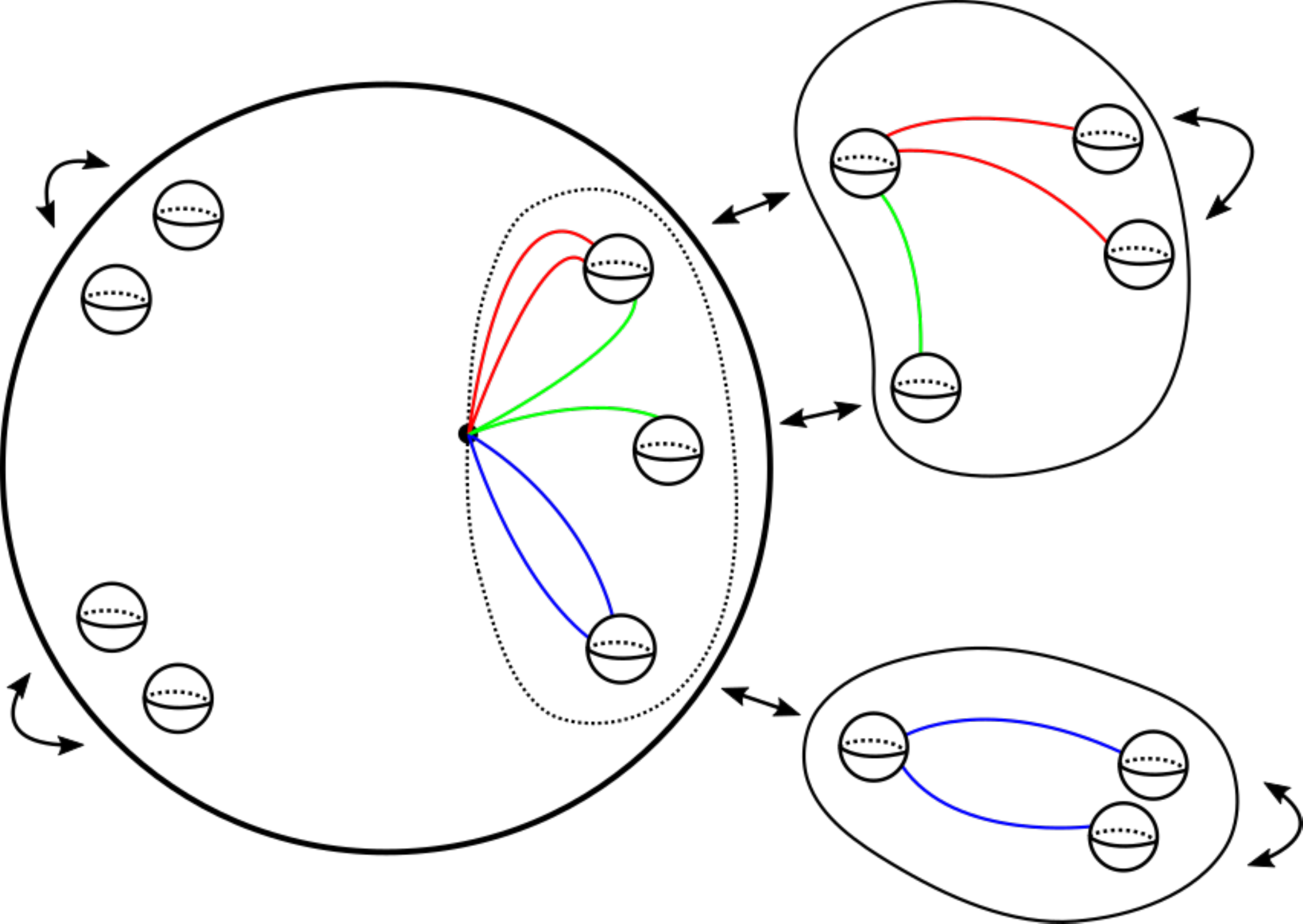}
  \caption{$\XX{2}{3}$ embedded inside $\XXnobound{5}$. For clarity, $x_1$ and $x_2$ are not shown, but they lie entirely on the opposite side of $\Sigma$ from $x_3$, $x_4$, and $x_5$.}
  \label{injsetup}

\end{figure}
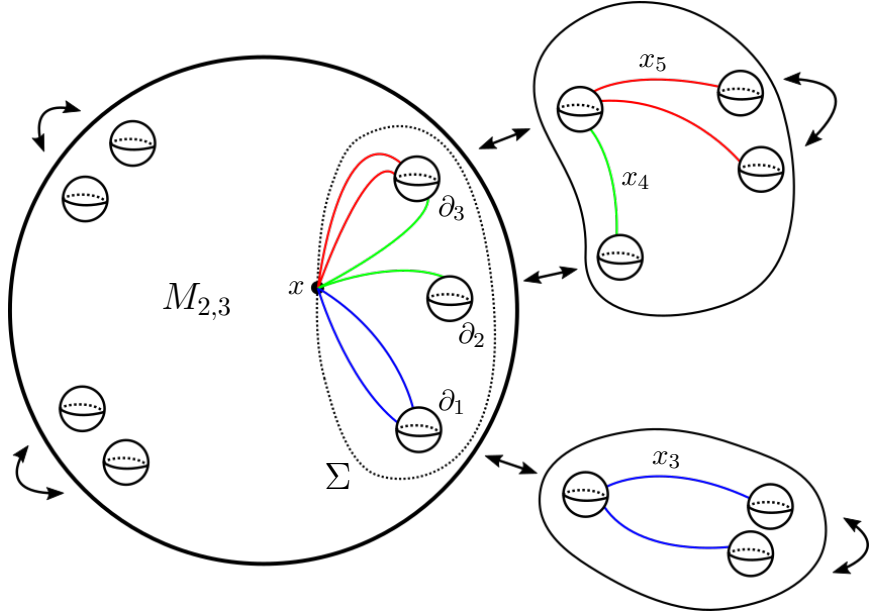

\underline{\textbf{Case 2:}} Next, suppose that $\iota: \XX{n}{b} \hookrightarrow \XXnobound{m}$ is an embedding, where $b > 1$. Let $\partial_1, \ldots, \partial_b$ be the boundary components of $\XX{n}{b}$. Let $\Sigma \subset \XX{n}{b}$ be a 2-sphere which separates $\XX{n}{b}$ into $\XX{n}{1}$ and $\XX{0}{b+1}$ (see Figure \ref{injsetup}). Then we have a composition of inclusions
\begin{equation*}
\XX{n}{1} \hookrightarrow \XX{n}{b} \hookrightarrow \XXnobound{m}. 
\end{equation*}
Let $\kappa_*: \Outb{n}{1} \to \Outb{n}{b}$ be the map induced by inclusion. By the preceding case, $\iota_* \circ \kappa_*$ is injective. Let $f \in \Outb{n}{b}$, and suppose that $\iota_*(f) = \id$. By repeated applications of the Birman exact sequence (Theorem \ref{birmanout}), $f$ has the form $f = p_1p_2 \cdots p_b \cdot \kappa_*(g)$, where $g \in \Outb{n}{1} \cong \Aut(F_n)$ and $p_j \in \Outb{n}{1}$ is a boundary drag of $\partial_j$ along a loop $\beta_j$. Fix a basepoint $x \in \Sigma$, and let $\gamma_j \in \pi_1(\XX{n}{b}, x)$ be representative of the free homotopy class of $\beta_j$. Choose a free basis $\{x_1, \ldots, x_n\}$ for $\pi_1(\XX{n}{1}, x)$. Extend this to a free basis $\{x_1, \ldots, x_m\}$ for $\pi_1(\XXnobound{m}, x)$ such that for each $i > n$, the loop $x_i$ intersects the set $\bigcup_{j=1}^b \partial_j$ exactly twice: once when exiting $\XX{n}{b}$, and once when re-entering (see Figure \ref{injsetup}). For $i > n$, let $\partial_{\ell(i)}$ be the boundary component through which $\alpha_i$ leaves $\XX{n}{b}$, and let $\partial_{r(i)}$ be the boundary component through which it returns. Then $\iota_*(f)$ is the class of the automorphism $\varphi \in \Aut(F_m)$ given by
\begin{equation*}
  \varphi: \begin{cases}
    x_i \mapsto g(x_i) & \text{ for } 1 \leq i \leq n \\
    x_i \mapsto \gamma_{\ell(i)} x_i \gamma_{r(i)}^{-1} & \text{ for } n < i \leq m.
  \end{cases}
\end{equation*}
By assumption, this automorphism is an inner automorphism. Suppose that $\varphi$ conjugates by a reduced word $w$ in the $x_i$. Since $g$ is an automorphism of $\langle x_1, \ldots, x_n \rangle \subset F_m$, it follows that $w \in \langle x_1, \ldots x_n \rangle$. We will show that this implies that $f$ is trivial by induction on the reduced word length of $w$. 

For the base case, suppose that the word length of $w$ is 0. Then $w$ and $\varphi$ are both trivial. Since $\iota_* \circ \kappa_*$ is injective, $g$ is trivial as well. Suppose now that some $\gamma_j$ is non-nullhomotopic. Since no component of $\XXnobound{m} \setminus \interior(\XX{n}{b})$ is a disk, there exists some $x_i$ which passes through $\partial_j$, where $i > n$. In other words, either $\ell(i) = j$ or $r(i) = j$. This is a contradiction because then $\varphi(x_i) = \gamma_{\ell(i)} x_i \gamma_{r(i)}^{-1} \neq x_i$. Thus, all $\gamma_j$ are nullhomotopic, and so $f$ is trivial. This completes the base case. 

Next, suppose that $w$ has positive word length, and let $x_i^{\pm 1}$ be the last letter in the reduced form of $w$. Then, $w = w'x_i^{\pm 1}$, where the length of $w'$ is less than that of $w$. To avoid notational complexity, we will assume that $x_i^{\pm 1} = x_1$, but the same argument works for any other $x_i$. Consider the element
\begin{equation*}
h := \HD_{21}\HD_{31} \cdots \HD_{n1} \cdot q_1 \cdots q_b \in \Outb{n}{b},
\end{equation*}
where $\HD_{i1}$ is the handle drag of the $i$-th handle about the first handle (see Section \ref{generatorssection}) and $q_j$ is obtained by dragging $\partial_j$ about a loop in the free homotopy class of $x_1$. By construction, $\iota_*(h) \in \Out(F_m)$ is the class of the automorphism which conjugates by $x_1$. Therefore, $\iota_*(h^{-1} f)$ is the class of the automorphism which conjugates by $w'$. By our induction hypothesis, this implies that $h^{-1}f$ is trivial.

\begin{claim}
The element $h$ is trivial. 
\end{claim}

\begin{proof}
Let $\Sigma' \subset \XX{n}{b}$ be a 2-sphere which separates $\XX{n}{b}$ into $\XX{1}{1}$ and $\XX{n-1}{b+1}$, where the $\XX{1}{1}$ is the handle containing $x_1$. Let $\lambda_*: \Outb{1}{1} \to \Outb{n}{b}$ be the map induced by this inclusion. Notice that $h = \lambda_*(q)$, where $q \in \Outb{1}{1}$ drags the boundary component of $\XX{1}{1}$ about the nontrivial loop in the positive direction. We saw in the proof of Lemma \ref{injlemma} that $\Outb{1}{1} \cong \mathbb{Z}/2$, and the nontrivial element acts on $\pi_1(\XX{1}{1})$ by inversion. However, the element $q$ acts trivially on $\pi_1(\XX{1}{1})$, and is thus trivial itself. It follows that $h$ is trivial as well. 
\end{proof}


Combining the claim with the fact that $h^{-1}f$ is trivial, we find that $f$ is trivial. This completes the induction, and so we conclude that $\iota_*$ is injective. 

\end{proof}

\section{Realizing homology classes as spheres}\label{homologybasessection} 

In this section, we prove a result used in the proof of Lemma \ref{complexiso} which involves realizing bases of $H_2(\XXnobound{n})$ as collections of 2-spheres. Recall that $H_2(\XXnobound{n}) = \mathbb{Z}^n$. This identification induces a homomorphism $\eta: \Mod(\XXnobound{n}) \to \GL_n(\mathbb{Z})$ which takes a mapping class to its action on homology.

\begin{lemma}\label{surjlemma}
  The map $\eta: \Mod(\XXnobound{n}) \to \GL_n(\mathbb{Z})$ is surjective. 
\end{lemma}

\begin{proof}
  First, notice that $H^1(\XXnobound{n}) = \mathbb{Z}^n$. This identification also induces a homomorphism $\eta': \Mod(\XXnobound{n}) \to \GL_n(\mathbb{Z})$ which is well-known to be surjective. Indeed, this map factors as
  \begin{equation*}
    \Mod(\XXnobound{n}) \overset{q}{\to} \Out(F_n) \overset{\varphi}{\to} \GL_n(\mathbb{Z}),
  \end{equation*}
  where $q$ is the quotient map, and $\varphi$ sends an automorphism class to its action on $H^1$. Therefore, if we choose our identifications $H^1(\XXnobound{n}) = \mathbb{Z}^n$ and $H_2(\XXnobound{n}) = \mathbb{Z}^n$ to agree with Poincar\'{e} duality, then $\eta$ and $\eta'$ are the same map. Thus, $\eta$ is surjective. 
\end{proof}

\begin{lemma}\label{realizationlemma}
  Let $\{ v_1, \ldots, v_n \}$ be a basis for $H_2(\XXnobound{n}) = \mathbb{Z}^n$, and let $A = \{ S_1, \ldots, S_\ell\}$ be a collection of disjoint embedded oriented 2-spheres in $\XXnobound{n}$ which satisfy $[S_j] = v_j$ for $1 \leq j \leq \ell$. Then $A$ can be extended to a collection $\overline{A} = \{S_1, \ldots, S_n\}$ of disjoint embedded oriented 2-spheres such that $[S_j] = v_j$ for $1 \leq j \leq n$. 
\end{lemma}
 
\begin{proof}
  We will induct on $n$. The base case $n=0$ is trivial. So assume $n > 0$, and let $\{v_1, \ldots, v_n\}$ and $A = \{ S_1, \ldots, S_\ell\}$ be as stated. There are two cases.

  First, suppose that $\ell = 0$. If we identity $H_2(\XXnobound{n})$ with $\mathbb{Z}^n$, then by Lemma \ref{surjlemma} the resulting map $\eta: \Mod(\XXnobound{n}) \to \GL_n(\mathbb{Z})$ is surjective. Choose any collection $\Sigma_1, \ldots, \Sigma_n \subset \XXnobound{n}$ of disjoint non-nullhomotopic embedded 2-spheres. Then $\{[\Sigma_1], \ldots, [\Sigma_n]\}$ is a basis for $H_2(\XXnobound{n})$. Since $\GL_n(\mathbb{Z})$ acts transitively on ordered bases of $\mathbb{Z}^n$ and the map $\eta$ is surjectve (Lemma \ref{surjlemma}), there exists some $\mathfrak{f} \in \Mod(\XXnobound{n})$ such that $\eta(\mathfrak{f}) \cdot [\Sigma_j] = v_j$ for all $1 \leq j \leq n$. In other words, $[\mathfrak{f}(\Sigma_j)] = v_j$, and so $\{\mathfrak{f}(\Sigma_1), \ldots, \mathfrak{f}(\Sigma_n)\}$ is the desired collection of spheres. 

  Next, suppose that $\ell > 0$. Splitting $\XXnobound{n}$ along $S_1$ gives an embedding $\iota: \XX{n-1}{2} \hookrightarrow \XXnobound{n}$. Notice that the induced map $\iota_H: H_2(\XX{n-1}{2}) \to H_2(\XXnobound{n})$ is an isomorphism. Let $w_j = \iota_H^{-1}(v_j)$ for $1 \leq j \leq n$, and let $\partial$ and $\partial'$ be the boundary components of $\XX{n-1}{2}$. Capping the two boundary components of $\XX{n-1}{2}$ with disks $D$ and $D'$, we get another embedding $\iota': \XX{n-1}{2} \hookrightarrow \XXnobound{n-1}$. This embedding induces a surjective map $\iota_H': H_2(\XX{n-1}{2}) \to H_2(\XXnobound{n-1})$ whose kernel is generated by $[\partial]$. Let $w_j' = \iota_H'(w_j)$ for $2 \leq j \leq n$, and let $S_k' = \iota'(S_k)$ for $2 \leq k \leq \ell$. By our induction hypothesis, we can extend the collection $\{S_2', \ldots, S_\ell'\}$ to a collection $\{S_2', \ldots, S_{n-1}'\}$ of disjoint embedded oriented 2-spheres in $\XXnobound{n-1}$ such that $[S_j'] = w_j'$ for $2 \leq j \leq n$. Moreover, since the disks $D$ and $D'$ used to cap the boundary components of $\XX{n-1}{2}$ are contractible, we may isotope $S_{\ell+1}', \ldots, S_{n-1}'$ such that they are disjoint from $D$ and $D'$. Let $S_j = (\iota')^{-1}(S_j')$ for $\ell + 1 \leq j \leq n$. If $[S_k] = w_k$ for all $k$, then $\{S_1, \ldots, S_n\}$ is the desired collection, and we are done. However, since the kernel of $\iota_H'$ is generated by $[\partial]$, we have 
  \begin{equation*}
    [S_k] = w_k + c_k[\partial],
  \end{equation*}
  where $c_k \in \mathbb{Z}$. Note that $c_k = 0$ for $2 \leq k \leq \ell$. To fix this, we may surger parallel copies of $\partial$ or $\partial'$ onto $S_k$ such that it has the correct homology class. The process is as follows (see Figure \ref{surgerspheres}):
  \begin{enumerate}[(i)]
    \item If $c_k > 0$, take $c_k$ parallel copies of $\partial'$, which we denote by $\partial_1, \ldots, \partial_{c_k}$. If instead $c_k < 0$, take $\partial_1, \ldots, \partial_{c_k}$ to be parallel copies of $\partial$. Order the $\partial_j$ such that $\partial_1$ is furthest from its respective boundary component, then $\partial_2$, and so on.
    \item Let $\gamma_1$ be a properly embedded arc connecting the positive side of $S_k$ to $\partial_1$ which does not intersect any of the other $S_j$ or $\partial_j$. 
    \item Surger $S_k$ and $\partial_1$ together via a tube running along $\gamma_1$.
    \item Repeat steps (ii) and (iii) for the remaining $\partial_j$. 
  \end{enumerate}
  Once we have carried out this process for all the $S_k$, we will have obtained a collection collection $\{S_2, \ldots, S_n\}$ of spheres whose homology classes are exactly $w_2, \ldots, w_n$. Thus, $\{S_1, \ldots, S_n\}$ is the desired collection of 2-spheres. 
\end{proof}

\begin{figure}[H]
  \centering
  \labellist
  \scriptsize\hair 2pt
  \pinlabel $S_k$ at 135 110
  \pinlabel $\gamma_1$ at 125 70
  \pinlabel $\partial_1$ at 100 45
  
  \pinlabel $\gamma_2$ at 330 120
  \pinlabel $S_k$ at 330 80
  
  \pinlabel $S_k$ at 530 80
  \endlabellist
  \includegraphics[width=\textwidth]{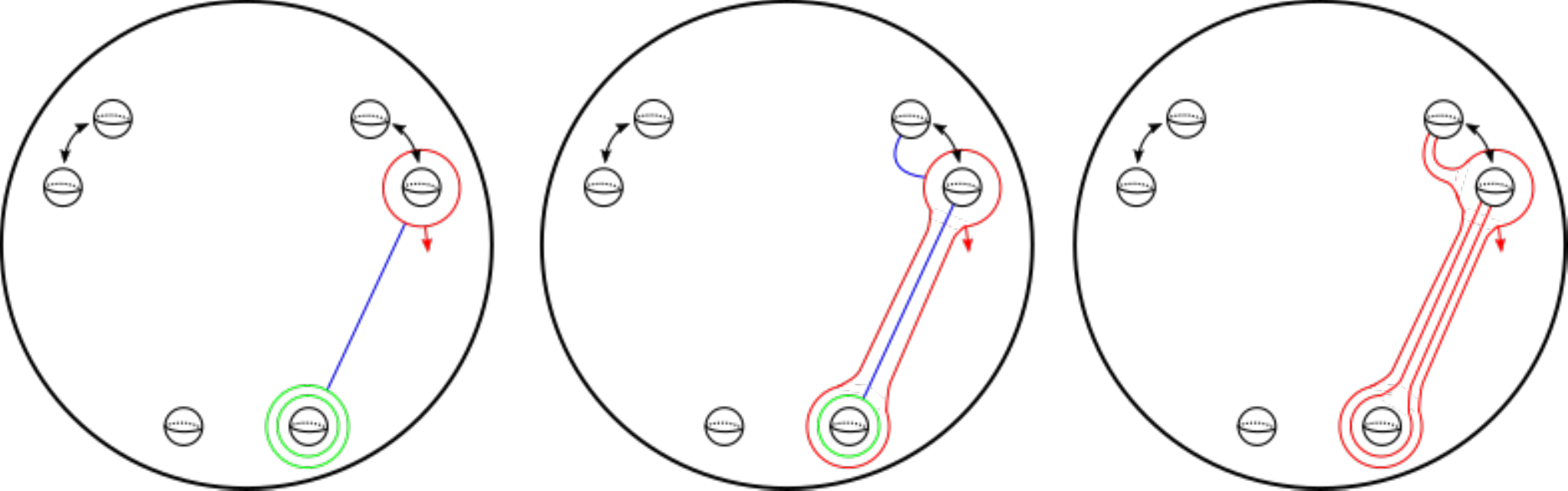}
  \caption{Surgering boundary spheres onto $S_k$.}
  \label{surgerspheres}
\end{figure}
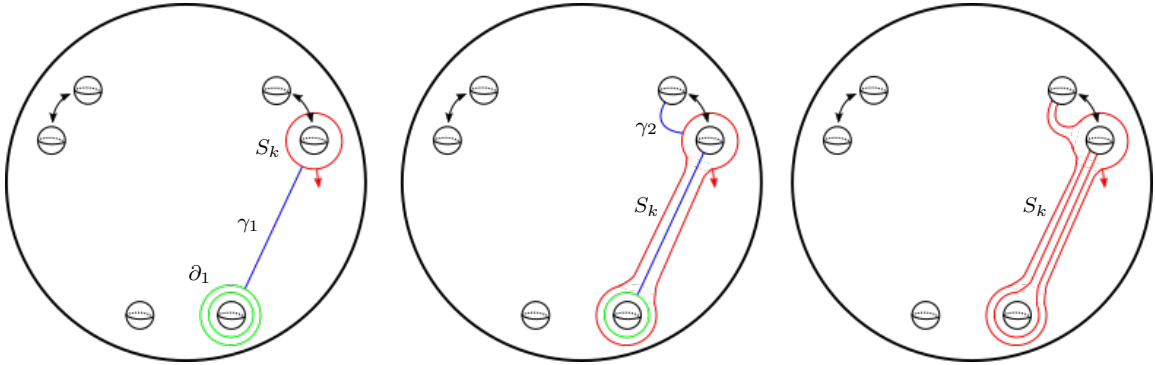

\bibliography{cutpastian}
\bibliographystyle{plain}
\end{document}